\documentclass[11pt,reqno]{amsart}

\usepackage{amsmath,amssymb,mathrsfs,amsthm}
\usepackage{graphicx,cite,cases}
\usepackage{xcolor}
\newtheorem{theorem}{Theorem}[section]
\newtheorem{corollary}[theorem]{Corollary}

\newtheorem{proposition}[theorem]{Proposition}

\newtheorem{definition}{Definition}[section]

\numberwithin{equation}{section}

\begin{document}

\title[nonradiating sources]{Nonradiating sources of the biharmonic wave equation}

\author{Peijun Li}
\address{Department of Mathematics, Purdue University, West Lafayette, Indiana
47907, USA.}
\email{lipeijun@math.purdue.edu}

\author{Jue Wang}
\address{School of Mathematics, Hangzhou Normal University,  Hangzhou 311121, China}
\email{wangjue@hznu.edu.cn}

\thanks{The first author is supported in part by the NSF grant DMS-2208256. The second author is supported by NSFC (No.12371420), the Natural Science Foundation of Zhejiang Province (No.LY23A010004) and the Scientific Research Starting Foundation (No.2022QDL017).}



\begin{abstract}
This paper offers an extensive exploration of nonradiating sources for the two- and three-dimensional biharmonic wave equations. Various equivalent characterizations are derived to reveal the nature of a nonradiating source. Additionally, we establish the connection between nonradiating sources in the biharmonic wave equation and those in the Helmholtz equation as well as the modified Helmholtz equation. Several illustrative examples are explicitly constructed to showcase the existence of nonradiating sources. One significant implication of the existence of nonradiating sources is that it undermines the uniqueness of the inverse source problem when utilizing boundary data at a fixed frequency.
\end{abstract}

\maketitle

\section{Introduction}\label{S:in}

In the study of wave propagation, two distinct types of wave sources come to light: radiating and nonradiating. Radiating sources emit waves that disperse into free space and can travel indefinitely. In contrast, nonradiating sources generate waves within a specific medium, which remain confined to a particular region in space without extending into free space. The exploration of nonradiating sources holds considerable importance across multiple domains of physics and engineering, such as geological exploration, electromagnetic detection, and stealth technology. Moreover, it plays an important role in tackling challenges posed by the inverse source problem.

Extensive research efforts have been devoted to exploring the mathematical and computational aspects of the inverse source problem concerning acoustic, electromagnetic, and elastic wave phenomena. From a mathematical perspective, the matter of uniqueness represents a substantial challenge when addressing inverse source problems. It is essential to emphasize that the inverse source problem at a specific frequency is inherently ill-posed, primarily due to the existence of nonradiating sources. Consequently, attempting to deduce the wave source uniquely based solely on a single far-field pattern or boundary data at a fixed frequency becomes an infeasible endeavor \cite{EV2009}. The issue of nonuniqueness was explored in \cite{BC1977} with regard to the inverse source problem in the fields of acoustics and electromagnetics. This work notably introduced the concept of nonradiating sources. The work presented in \cite{AM2006} addressed the inverse source problem associated with Maxwell's equations. The study demonstrated that the inverse problem of determining a volume current density from surface measurements lacks a unique solution. This nonuniqueness in solutions was clarified through the variational approach. In \cite{B2018, B2019}, nonradiating sources were investigated under a fixed frequency condition for both the acoustic and elastic wave equations, utilizing the enclosure method. The findings of the study established that nonradiating sources exhibiting a convex or non-convex corner or edge along their boundary are required to undergo vanishing behavior in those regions. In order to address the challenge posed by nonuniqueness, the work presented in \cite{BLLT2011} provides a comprehensive review of the subject matter concerning inverse scattering problems, centering on the utilization of multi-frequency boundary measurements as a strategic approach. Moreover, originating from the work in \cite{BLT2010}, the investigation into stability has received substantial attention in inverse source problems, particularly when considering multiple frequency data \cite{BLZ2020, CIL2016, LY2017, LZZ2020}. Recently, research has undertaken the exploration of the more intricate inverse source problems associated with stochastic wave equations, as exemplified by \cite{BCL2016, LLW2022} and the references cited therein.

The examination of scattering problems associated with biharmonic waves has attracted substantial interest, largely attributed to their essential applications in the field of thin plate elasticity \cite{SWW2012, FGE2009}. Nevertheless, it is worth highlighting that the inverse source problems for biharmonic wave equations demonstrate limited findings in contrast to the more extensively studied acoustic, elastic, and electromagnetic wave equations. This limitation arises due to the complex nature of solutions inherent to biharmonic equations. As a result, certain methodologies that have proven effective for second-order wave equations are rendered inadequate in this context \cite{APS2000, KI1988, YY2014}.
For the latest advancements within this research domain, we direct attention to \cite{GX2019, LW2023, LW2022, LYZ2021, LYZ2021IP}, as well as the related references presented therein.

Inspired by the research conducted by \cite{BC1977}, the present paper aims to explore the detailed characterizations of nonradiating sources for the biharmonic wave equation in both two and three dimensions. We establish a set of necessary and sufficient conditions that determine the nonradiating characteristic of a source. This is achieved by employing the integral representation of the field and expanding the fundamental solution in special functions. Additionally, we deduce a fundamental identity that establishes a direct connection between the integral transform of the source distribution and the ensuing near-field data. This outcome provides us with the capability to extract the integral transform of the source distribution by utilizing near-field data acquired from the boundary. Moreover, we conduct an extensive discussion concerning the integral equation inherent to the problem and examine the nonuniqueness associated with the inverse source problem. Our investigation reveals the presence of a nontrivial null space within the integral operator, thereby illuminating the ill-posed nature of the inverse source problem. This result offers invaluable insights for future research endeavors and practical applications. Finally, we construct several illustrative examples explicitly to demonstrate the existence of nonradiating sources.

The structure of this paper is outlined as follows. Section \ref{S:bwe} introduces the biharmonic wave equation and the problem formulation. Sections \ref{S:2d} and \ref{S:3d} establish distinct yet equivalent characterizations for the identification of nonradiating sources in two and three dimensions, respectively. In Section \ref{S:ex}, various instances of nonradiating sources are presented through illustrative examples. The conclusion and broader insights are offered in Section \ref{S:co}.

\section{The biharmonic wave equation}\label{S:bwe}

The biharmonic wave equation characterizes wave behavior by considering both the curvature and deflection of the wavefront. Mathematically, the time-harmonic biharmonic wave equation can be expressed as
\begin{align}\label{bhwEq}
\Delta^2 u -\kappa^4 u =-f\quad\text{in} ~ \mathbb{R}^d,
\end{align}
where the spatial dimension is denoted as $d=2$ or $d=3$, the operator $\Delta^2$ stands for the biharmonic operator, capturing effects of curvature and deformation, the scalar function $u$ describes the displacement amplitude of the wave, the parameter $\kappa>0$ denotes the wavenumber, related to oscillation frequency and wave propagation speed, the term $f$ represents the source function, which is assumed to have a compact support $D$ contained in the ball $B_R=\{x\in\mathbb R^d: |x|<R\}$. To ensure the well-posedness of the problem, the Sommerfeld radiation condition is imposed on both the wave field $u$ and its Laplacian $\Delta u$ (cf. \cite{TS2018}):
\begin{align}\label{SRC}
\lim_{r\to\infty} r^{\frac{d-1}{2}}(\partial_r u-{\rm i}\kappa u)=0,\quad\quad 
\lim_{r\to\infty} r^{\frac{d-1}{2}}(\partial_r \Delta u-{\rm i}\kappa\Delta u)=0,\quad\quad  r=|x|,
\end{align}
uniformly in all directions ${\hat x} =  x/|x|$. 

The analysis of far-field patterns has received significant emphasis in the field of inverse scattering theory. Considering a solution $u$ of \eqref{bhwEq}--\eqref{SRC}, it has the asymptotic behavior 
\begin{align*}
u(x)=-\frac{\mu_d}{8\kappa^2}\frac{e^{{\rm i}\kappa|x|}}{(\pi|x|)^{\frac{d-1}{2}}}\bigg[u_{\infty}({\hat x})+O\left(\frac{1}{|x| }\right)\bigg],\quad \ |x|\to \infty,
\end{align*}
where $u_\infty$ is called the far-field pattern of the wave field $u$ and 
\begin{equation*}
 \mu_d=\left\{
\begin{aligned}
&\left(\frac{2}{\kappa}\right)^{1/2}e^{{\rm i}\frac{\pi}{4}}&\quad\text{if } ~ d=2,\\
&1 &\quad \text{if } ~d=3. 
\end{aligned}
\right.
\end{equation*}

Let $G(x, y)$ be the Green's function of the biharmonic wave equation, which satisfies 
\begin{align}\label{BiFSE}
\Delta^2 G(x, y)-\kappa^4 G(x, y)=-\delta(x-y),\quad x, y\in\mathbb R^d, ~ x\neq y.
\end{align}
It can be verified that $G(x, y)$ is given by 
\begin{align}\label{BiFS}
G(x, y)=-\frac{1}{2\kappa^2}\left(\Phi_{\rm H}(x, y)-\Phi_{\rm M}(x, y)\right),
\end{align}
where $\Phi_{\rm H}(x, y)$ and $\Phi_{\rm M}(x, y)$ represent the Green's functions corresponding to the Helmholtz equation and the modified Helmholtz equation, respectively, and they satisfy 
\begin{equation}\label{Phi_HM}
\Delta \Phi_{\rm H}(x, y)+\kappa^2 \Phi_{\rm H}(x, y)=-\delta(x- y),\quad
\Delta \Phi_{\rm M}(x, y)-\kappa^2 \Phi_{\rm M}(x, y)=-\delta(x- y).
\end{equation}

Using \eqref{bhwEq} and \eqref{BiFSE}, and noting that the source has a compact support $D$ contained in $B_R$,  we obtain that the solution to \eqref{bhwEq}--\eqref{SRC} can be written as
\begin{align}\label{Bihs}
u(x)=\int_{B_R}G(x, y)f(y){\rm d}y,\quad  x\in\mathbb R^d. 
\end{align}

Nonradiating sources, referring to sources where the field remains consistently zero beyond a finite region, have a strong association with the nonuniqueness issues encountered in the inverse source problems. The inverse source problem aims to determine the source function $f$ by utilizing measurements associated with the wave field $u$ on the boundary 
$\partial B_R$. Moving forward, our focus is directed towards the introduction and subsequent examination of nonradiating sources for the biharmonic wave equation. 

Let us present a formal definition of a nonradiating source in the context of the biharmonic wave equation.

\begin{definition}\label{def}
Consider a source denoted as $f$, having a compact support designated as $D$, which is entirely confined within the ball $B_R$. A source is classified as nonradiating if the solution $u$ to equation \eqref{Bihs} is identically zero outside of the ball $B_R$.
\end{definition}

In the following two sections, we examine the two- and three-dimensional problems, respectively. In each scenario, we provide distinct yet equivalent characterizations of nonradiating sources. We then proceed to investigate these characterizations utilizing near-field data, which is followed by a discussion of the nonuniqueness inherent in the inverse source problem. Our analysis relies on the utilization of expansions involving Bessel functions of various types within the framework of the fundamental solutions for the Helmholtz and modified Helmholtz equations. This method finds its roots in the work conducted by \cite{BC1977}, where it was employed to analyze nonradiating sources within the context of the three-dimensional Helmholtz equation.

\section{The two-dimensional problem}\label{S:2d}

In this section, we discuss different characterizations of nonradiating sources for the two-dimensional biharmonic wave equation.

\subsection{Characterizations of nonradiating sources}

In two dimensions, the explicit expressions for the Green's functions corresponding to the Helmholtz and modified Helmholtz equations, as defined in \eqref{Phi_HM}, are given by
\begin{align*}
\Phi_{\rm H}(x, y)=\frac{\rm i}{4}H_{0}^{(1)}(\kappa|x-y|),\quad
\Phi_{\rm M}(x, y)=\frac{\rm i}{4}H_{0}^{(1)}({\rm i}\kappa|x- y|),
\end{align*}
where $H_{0}^{(1)}$ is the Hankel function of the first kind with order $0$.

Define a pair of auxiliary functions 
\begin{align}\label{Binse-0}
f_{\rm H}(x)=-\int_{B_{R}}\Phi_{\rm H}(x, y)f(y){\rm d}y,\quad
f_{\rm M}(x)=-\int_{B_{R}}\Phi_{\rm M}(x, y)f(y){\rm d}y,\quad x\in\mathbb{R}^2.
\end{align}
It is clear to note that $f_{\rm H}$ and $f_{\rm M}$ are the radiation solutions to the Helmholtz and modified Helmholtz equations, respectively, i.e., they satisfy 
\begin{equation*}
 \Delta f_{\rm H}+\kappa^2 f_{\rm H}=f,\quad \Delta f_{\rm M}-\kappa^2 f_{\rm M}=f. 
\end{equation*}
Furthermore, both solutions satisfy the Sommerfeld radiation condition.

Combining \eqref{BiFS}, \eqref{Bihs}, and \eqref{Binse-0}, we obtain 
\begin{align}\label{Binse}
u(x)&=\int_{B_{R}}G(x,y)f(y){\rm d}y
=-\frac{1}{2\kappa^2}\int_{B_{R}}\big(\Phi_{\rm H}(x,y)-\Phi_{\rm M}(x,y)\big)f(y){\rm d}y\nonumber\\
&=\frac{1}{2\kappa^2}\big(f_{\rm H}(x)-f_{\rm M}(x)\big).
\end{align}

First, we establish the connection between a nonradiating source associated with the biharmonic wave equation and a nonradiating source related to both the Helmholtz equation and the modified Helmholtz equation through the subsequent theorem. 

\begin{theorem}\label{2DThe-nrs}
The source $f$ is nonradiating for the biharmonic wave equation if and only if $f$ is nonradiating for both the Helmholtz equation and the modified Helmholtz equation, i.e., $f_{\rm H}$ and $f_{\rm M}$ defined in \eqref{Binse-0} are identically zero outside of the ball $B_R$.
\end{theorem}

\begin{proof}
If $f$ is nonradiating for both the Helmholtz equation and the modified Helmholtz equation, then we have from \eqref{Binse-0} that $f_{\rm H}(x)= f_{\rm M}(x)=0$ for $|x|>R$. It follows from \eqref{Binse} that $u(x)=0$ for $|x|>R$, which shows that $f$ is radiating for the biharmonic wave equation. 

Conversely, if the source $f$ is nonradiating concerning the biharmonic wave equation, then based on \eqref{Binse}, it can be deduced that $u(x)=0$ for $|x|>R$. Upon considering \eqref{Binse}, two scenarios emerge: (1) where $f_{\rm H}(x)$ and $f_{\rm M}(x)$ are both non-zero for $|x|> R$, yet satisfy $f_{\rm H}(x)=f_{\rm M}(x)$ for all $x$ satisfying $|x|>R$; and (2) where $f_{\rm H}(x)=f_{\rm M}(x)=0$ for all $|x|>R$. We demonstrate that the former scenario is impossible, while the latter scenario indicates that $f$ is nonradiating in relation to both the Helmholtz equation and the modified Helmholtz equation.

When $f_{\rm H}(x)\neq 0, f_{\rm M}(x)\neq 0$ for $|x| > R$, it is clear to note that the source $f(x)$ does not qualify as a nonradiating source for either the Helmholtz equation or the modified Helmholtz equation. By the asymptotic expansion
\begin{align}\label{2Dax-y}
|x-y|= |x|- {\hat x} \cdot y+O\left(\frac{1}{|{x}|}\right),\quad |x|>|y|, ~ |x|\to \infty,
\end{align}
it can be verified for $|x|>|y|, |x|\to \infty$ that the Green's functions $\Phi_{\rm H}$ and $\Phi_{\rm M}$ admit the asymptotic expansions
\begin{align}
\Phi_{\rm H}(x, y)=\frac{1}{4}\left(\frac{2}{\pi\kappa| {x}|}\right)^{1/2}
e^{{\rm i}(\kappa|x|+\frac{\pi}{4})}\bigg[e^{-{\rm i}\kappa {\hat x} \cdot y}+O\left(\frac{1}{|{x}| }\right)\bigg] \label{2Dasb}
\end{align}
and
\begin{align}
\Phi_{\rm M}(x, y)=\frac{1}{4}\left(\frac{2}{{\rm i}\pi\kappa| {x}|}\right)^{1/2}e^{(-\kappa|x|+{\rm i}\frac{\pi}{4})}~\bigg[e^{\kappa {\hat x} \cdot y}+O\left(\frac{1}{|{x}|}\right)\bigg].\label{2Dasb-1}
\end{align}
Hence, we have
\begin{align}\label{2DFSab}
G(x,y)
&=-\frac{1}{8\kappa^2}e^{{\rm i}(\kappa|x|+\frac{\pi}{4})}\left(\frac{2}{\pi\kappa| {x}|}\right)^{1/2}
\bigg[e^{-{\rm i}\kappa {\hat x} \cdot y}+O\left(\frac{1}{|{x}|}\right)\bigg],\quad\ |x|>|y|, ~ |x|\to \infty.
\end{align}
Combining \eqref{Binse} and \eqref{2DFSab} yields 
\begin{align*}
u(x)=-\frac{1}{8\kappa^2}e^{{\rm i}(\kappa|x|+\frac{\pi}{4})}\left(\frac{2}{\pi\kappa| {x}|}\right)^{1/2}
\bigg[\int_{B_R}e^{-{\rm i}\kappa {\hat x} \cdot y}f(y){\rm d}y+O\left(\frac{1}{| {x}| }\right)\bigg],\quad  |x|\to \infty,
\end{align*}
which implies that $u(x)\not\equiv 0$ for $|x| > R$. Consequently, the first scenario is thereby eliminated, leading to the conclusion of the proof. 
\end{proof}

In order to gain a more comprehensive understanding of the nonradiating source, we undertake an additional analysis, exploring how its reliance on the projection of specific coefficients influences the nonradiating source.

Let $H_n^{(1)}=J_n+{\rm i}Y_n$ be the Hankel function of the first kind with order $n$, where $J_n$ and $Y_n$ are the Bessel and Neumann functions of order $n$, respectively. When $|x|>|y|$, we utilize Graf's addition theorem (cf. \cite[$(3.65)$]{DR-2013}) to express the functions $\Phi_{\rm H}(x,y)$ and $\Phi_{\rm M}(x,y)$ as a combination of multipolar sources centered at the origin:
\begin{align}
\Phi_{\rm H}(x,y)&=\frac{\rm i}{4}H_{0}^{(1)}(\kappa|x-y|)
=\frac{\rm i}{4}\sum_{n=-\infty}^{\infty}H_{n}^{(1)}(\kappa|x|)e^{{\rm i} n {\rm arg}(x)}J_{n}(\kappa|y|)e^{-{\rm i} n {\rm arg}(y)},\label{Graf-1}\\
\Phi_{\rm M}(x,y)&=\frac{\rm i}{4}H_{0}^{(1)}({\rm i}\kappa|x-y|)
=\frac{\rm i}{4}\sum_{n=-\infty}^{\infty}H_{n}^{(1)}({\rm i}\kappa|x|) e^{{\rm i} n {\rm arg}(x)}J_{n}({\rm i}\kappa|y|) e^{-{\rm i} n {\rm arg}(y)},\label{Graf-2}
\end{align}
where ${\rm arg}(x)$ represents the counterclockwise oriented angle formed by the vector $(1, 0)$ and the vector $x$.
It is noteworthy to mention that in \eqref{Graf-2}, the Hankel function $H_n^{(1)}$ and the Bessel function $J_n$ with purely imaginary arguments can be substituted with the modified Hankel functions $K_n$ and $I_n$.

Assume $f\in L^2(B_R)$ and let 
\[
 f_n(r)=\frac{1}{2\pi}\int_{0}^{2\pi} f(x) e^{- {\rm i} n \theta}{\rm d}\theta.
\]
It is evident that the coefficients $f_n$ represent the Fourier coefficients of the function $f$ in the following expansion:
\[
 f(x)=\sum_{n=-\infty}^\infty f_n(r) e^{{\rm i} n\theta},\quad |x|<R. 
\]
By employing the Fourier coefficients $f_n$, we define the parameters $\alpha_n$ and $\beta_n$ as follows:
\begin{equation}\label{abn}
 \alpha_n=\int_{0}^{R}f_n(r)J_{n}(\kappa r)r{\rm d}r,\quad \beta_n=\int_{0}^{R}f_n(r)J_{n}({\rm i}\kappa r)r{\rm d}r.
\end{equation}

Presented below is an alternative characterization of nonradiating sources.

\begin{theorem}\label{2DThe-cn}
Assume that $f\in L^2(B_R)$. Then the source $f$ is nonradiating if and only if $\alpha_n=\beta_n=0$ for all $n\in \mathbb{Z}$.
\end{theorem}

\begin{proof}
Combining \eqref{Binse-0} and \eqref{Graf-1}, we obtain 
\begin{align}\label{fMs}
f_{\rm H}(x)
&=-\int_{B_{R}}\Phi_{\rm H}(x,y)f(y){\rm d}y
=-\frac{\rm i}{4}\sum_{n=-\infty}^{\infty}H_{n}^{(1)}(\kappa|x|)e^{{\rm i} n {\rm arg}(x)}\int_{B_{R}} f(y) J_{n}(\kappa|y|)e^{-{\rm i} n {\rm arg}(y)}{\rm d}y\nonumber\\
&=-\frac{\rm i}{4}\sum_{n=-\infty}^{\infty}H_{n}^{(1)}(\kappa|x|)e^{{\rm i} n {\rm arg}(x)}\int_{0}^{R}\int_{0}^{2\pi} f(y) J_{n}(\kappa r)e^{-{\rm i} n \theta}r{\rm d}\theta{\rm d}r\nonumber\\
&=-\frac{{\rm i}\pi}{2}\sum_{n=-\infty}^{\infty}\alpha_n H_{n}^{(1)}(\kappa|x|)e^{{\rm i} n {\rm arg}(x)},\quad |x| > R.
\end{align}

Similarly, we have from \eqref{Binse-0} and \eqref{Graf-2} that 
\begin{align}\label{fHs}
f_{\rm M}(x)
&=-\int_{B_{R}}\Phi_{\rm M}(x,y)f(y){\rm d}y
=-\frac{\rm i}{4}\sum_{n=-\infty}^{\infty}H_{n}^{(1)}({\rm i}\kappa|x|)e^{{\rm i} n {\rm arg}(x)}\int_{B_{R}} f(y) J_{n}({\rm i}\kappa|y|)e^{-{\rm i} n {\rm arg}(y)}{\rm d}y\nonumber\\
&=-\frac{\rm i}{4}\sum_{n=-\infty}^{\infty}H_{n}^{(1)}({\rm i}\kappa|x|)e^{{\rm i} n {\rm arg}(x)}\int_{0}^{R}\int_{0}^{2\pi} f(y) J_{n}({\rm i}\kappa r)e^{-{\rm i} n \theta}r{\rm d}\theta{\rm d}r\nonumber\\
&=-\frac{{\rm i}\pi}{2}\sum_{n=-\infty}^{\infty}\beta_n H_{n}^{(1)}({\rm i}\kappa|x|)e^{{\rm i} n {\rm arg}(x)},\quad |x| > R. 
\end{align}
The proof is completed by noting Theorem \ref{2DThe-nrs} in conjunction with \eqref{fMs}--\eqref{fHs}.
\end{proof}

If $f\in L^2(B_R)$, we can deduce from \cite[Theorem 8.2]{DR-2013} and \eqref{Binse-0} that both $f_{\rm H}$ and $f_{\rm M}$ belong to the space $H^2_{\rm loc}(\mathbb R^2)$. Consequently, the traces of the functions $f_{\rm H}|_{\partial B_R}$ and $f_{\rm M}|_{\partial B_R}$ can be found in the space $H^{\frac{3}{2}}(\partial B_R)$, and their respective normal derivatives, $\partial_\nu f_{\rm H}|_{\partial B_R}$ and $\partial_\nu f_{\rm M}|_{\partial B_R}$, are contained in $H^{\frac{1}{2}}(\partial B_R)$.

\begin{proposition}\label{2DThe-fHct}
Assume that $f\in L^2(B_R)$. Then $f_{\rm M}(x)|_{\partial B_R}=0$ if and only if $\beta_n=0$ for all $n\in \mathbb{Z}$.
\end{proposition}

\begin{proof}
Given that the source $f$ has a compact support $D$ confined within $B_R$, the function $f_{\rm M}$ satisfies the equation $\Delta f_{\rm M}-\kappa^2 f_{\rm M}=0$ in $\mathbb{R}^2\setminus \overline{B_R}$. By \eqref{Binse-0} and \eqref{2Dasb-1}, it becomes evident that $f_{\rm M}$ satisfies the Sommerfeld radiation condition. With the information that $f_{\rm M}(x)|_{\partial B_R}=0$ and considering the Sommerfeld radiation condition, we conclude that $f_{\rm M}(x)=0$ for $|x|>R$. It follows from \eqref{fHs} that $\beta_n=0$ holds true for all $n\in \mathbb{Z}$.

Conversely, based on the fact that $\beta_n=0$ for all $n\in \mathbb{Z}$ and \eqref{fHs}, it can be deduced that $f_{\rm M}(x)=0$ when $|x|>R$, which, with the information that $f_{\rm M}|_{\partial B_R}\in H^{\frac{3}{2}}(\partial B_R)$, leads us to the conclusion that 
$f_{\rm M}|_{\partial B_R}=0$.
\end{proof}

Similarly, the following proposition can be demonstrated.

\begin{proposition}\label{2DThefnf}
Assume that $f(x)\in L^2(B_R)$. Then $f_{\rm H}(x)|_{\partial B_R}=0$ if and only if $\alpha_n=0$ for all $n\in \mathbb{Z}$.
\end{proposition}

By combining Theorem \ref{2DThe-cn} with Propositions \ref{2DThe-fHct} and \ref{2DThefnf}, we formulate an equivalent characterization of nonradiating sources.

\begin{corollary}\label{2DThefnf-1}
The source $f\in L^2(B_R)$ is nonradiating if and only if $f_{\rm H}(x)|_{\partial B_R}=f_{\rm M}(x)|_{\partial B_R}=0$.
\end{corollary}

Subsequently, we move forward to establish an alternative characterization of nonradiating sources from the perspective of integral transforms.

First, we consider the Fourier transform of the source, i.e.,
\begin{align}\label{FEf}
\hat{f}(\xi)=\int_{B_{R}}f(x) e^{-{\rm i} \xi \cdot x}{\rm d}x,
\end{align}
where $\xi\in\mathbb{R}^2$ is the spatial frequency. Recall the Jacobi--Anger expansion for the plane wave (cf. \cite[$(3.89)$]{DR-2013}):
\begin{align}\label{PWJAE}
e^{{\rm i}\kappa{x}\cdot{ d}}
=\sum_{n=-\infty}^{+\infty}{\rm i}^{n}e^{-{\rm i} n{\rm arg}( d)}J_{n}(\kappa|x|)e^{{\rm i} n{\rm arg}(x)}
,\quad x\in\mathbb{R}^2,
\end{align}
where $ d\in\mathbb R^2$ is the unit vector of propagation direction. 

Using \eqref{PWJAE}, we have from a simple calculation that 
\begin{align}\label{PWJAE-1}
e^{-{\rm i}{ \xi}\cdot{x}}
=\overline{e^{{\rm i}{ \xi}\cdot{x}}}
=\sum_{n=-\infty}^{+\infty}(-{\rm i})^{n}e^{{\rm i} n{\rm arg}(\hat{ \xi})}J_{n}(|{ \xi}||x|)e^{-{\rm i} n{\rm arg}(x)}, 
\end{align}
where $\hat{ \xi}={ \xi}/|{ \xi}|$. Substituting \eqref{PWJAE-1} into \eqref{FEf} yields
\begin{align*}
\hat{f}(\xi)
&=\int_{B_{R}}f(x) \bigg(\sum_{n=-\infty}^{+\infty}(-{\rm i})^{n}e^{{\rm i} n{\rm arg}(\hat{ \xi})}J_{n}(|{ \xi}||x|)e^{-{\rm i} n{\rm arg}(x)}\bigg){\rm d}x
\nonumber\\
&=2\pi\sum_{n=-\infty}^{+\infty}(-{\rm i})^{n}e^{{\rm i} n{\rm arg}(\hat{ \xi})}\int_{0}^{R}J_{n}(|{ \xi}|r)r\bigg(\frac{1}{2\pi}\int_{0}^{2\pi}f(x) e^{-{\rm i} n\theta}{\rm d} \theta\bigg){\rm d}r
\nonumber\\
&=2\pi\sum_{n=-\infty}^{+\infty}(-{\rm i})^{n}e^{{\rm i} n{\rm arg}(\hat{ \xi})}\int_{0}^{R}f_n(r)J_{n}(|{ \xi}|r)r{\rm d}r,
\end{align*}
which implies, upon employing the definition of $\alpha_n$ in \eqref{abn}, that
\begin{align}\label{xik}
\hat{f}(\xi)
=2\pi\sum_{n=-\infty}^{+\infty}(-{\rm i})^{n}\alpha_n e^{{\rm i} n{\rm arg}(\hat{ \xi})},\quad|\xi|=\kappa.
\end{align}

It follows from \eqref{PWJAE} and \cite[Proposition 3.1]{WZC-2021} that 
\begin{align*}
e^{-\kappa{x}\cdot{ d}}
=e^{{\rm i}({\rm i}\kappa){x}\cdot{ d}}
=\sum_{n=-\infty}^{+\infty}{{\rm i}}^{n}e^{-{\rm i} n{\rm arg}( d)}J_{n}({\rm i}\kappa|x|)e^{{\rm i} n{\rm arg}(x)},\quad x\in\mathbb{R}^2. 
\end{align*}
Then, for $|x|=r\leq R$ and $0<| s|< \kappa+K$, where $s\in\mathbb R^2$ and $K>0$ is a bounded constant, we have
\begin{align}\label{Lp-4}
e^{-{ s}\cdot{x}}=e^{-{|x| s}\cdot{ {\hat x}}}
=e^{-{r s}\cdot{ {\hat x}}}=\sum_{n=-\infty}^{+\infty}{{\rm i}}^{n}e^{{\rm i} n{\rm arg}( s)}J_{n}({\rm i}| s|r)e^{-{\rm i} n{\rm arg}( {\hat x})}.
\end{align}

Define an integral transform
\begin{align}\label{Lp-de}
\check{f}(s)
=\int_{B_{R}}f(x) e^{- s \cdot x}{\rm d}x,\quad s\in\mathbb R^2, ~ 0<| s|< \kappa+K. 
\end{align}
Combining \eqref{Lp-4} and \eqref{Lp-de}, we obtain 
\begin{align*} 
\check{f}( s)
&=\int_{0}^{R}\int_{0}^{2\pi}rf(x)\bigg(\sum_{n=-\infty}^{+\infty}{{\rm i}}^{n}e^{{\rm i} n{\rm arg}( s)}J_{n}({\rm i}| s|r)e^{-{\rm i} n{\rm arg}( {\hat x})}\bigg){\rm d}r{\rm d} \theta\nonumber\\
&=2\pi\sum_{n=-\infty}^{+\infty}{{\rm i}}^{n}e^{{\rm i} n{\rm arg}( s)}\int_{0}^{R}rJ_{n}({\rm i}| s|r)\bigg(\frac{1}{2\pi}\int_{0}^{2\pi}f(x)e^{-{\rm i} n\theta}{\rm d} \theta\bigg){\rm d}r\nonumber\\
&=2\pi\sum_{n=-\infty}^{+\infty}{{\rm i}}^{n}e^{{\rm i} n{\rm arg}( s)}\int_{0}^{R}f_n(r)J_{n}({\rm i}| s|r)r{\rm d}r,\quad 0<| s|< \kappa+K,
\end{align*}
which, together with the definition of $\beta_n$ as given in \eqref{abn}, results in
\begin{align}\label{Lp-1-1}
\check{f}( s)
=2\pi\sum_{n=-\infty}^{+\infty}{{\rm i}}^{n}\beta_n e^{{\rm i} n{\rm arg}( s)},\quad
| s|=\kappa.
\end{align}

Therefore, utilizing Theorem \ref{2DThe-cn}, along with \eqref{xik} and \eqref{Lp-1-1}, we arrive at a different yet equivalent description of nonradiating sources.

\begin{theorem}\label{2DThe-2}
Assume that $f\in L^2(B_R)$. Then the source $f$ is nonradiating if and only if $\hat{f}(\xi)=\check{f}( s)=0$ when $|\xi|=|s|=\kappa$. 
\end{theorem}

\subsection{Near-field data}

We proceed to derive integral identities that establish connections between the integral transform of the source function and the near-field data $u$, $\partial_\nu u$, $\Delta u$, and $\partial_\nu \Delta u$ measured on the boundary $\partial B_R$.

By employing Green's theorem to the functions $u$ and $\Delta G$ in the domain $D$, with the help of
\eqref{BiFSE}, we get
\begin{align}\label{BiGTh}
&\int_{\partial B_R}\big[\Delta G(x,y) \partial_\nu u(y) -u(y) \partial_{\nu(y)} \Delta G(x,y)\big]{\rm d}{s_{{y}}}\notag\\
&=\int_{B_R}\big[ \Delta G(x,y)\Delta u(y)
-\kappa^4u(y)G(x,y)\big]{\rm d}{{y}}
+
\left\{\begin{array}{lll}
	u(x), &&x\in B_R,\\
	0,&& x\in\mathbb R^2\setminus\overline{B_R}. 
\end{array}\right.
\end{align}
where $\nu$ denotes the unit normal vector to the boundary $\partial B_R$, oriented outward from the interior of $B_R$.
Similarly, in the case of $\Delta u$ and $G$, we can deduce from Green's theorem, coupled with \eqref{bhwEq} and \eqref{BiFSE}, that
\begin{align}\label{BiGTh-1}
&\int_{\partial B_R}\big[\Delta u(y)\partial_{\nu(y)}  G(x,y) - G(x,y)\partial_\nu \Delta u(y)
\big]{\rm d}{s_{{y}}}\nonumber\\
&=\int_{B_R}\big[ \Delta G(x,y)\Delta u(y)-\kappa^4u(y)G(x,y)\big]{\rm d}{{y}}
+\int_{B_R}G(x,y)f(y){\rm d}{y}.
\end{align}
Subtracting \eqref{BiGTh-1} from \eqref{BiGTh} leads to 
\begin{align}\label{BiGTh-2}
&\int_{\partial B_R}\big[\Delta G(x,y)\partial_\nu u(y) -u(y)\partial_{\nu(y)} \Delta G(x,y)
\big]{\rm d}{s_{{y}}}\nonumber\\
&-\int_{\partial B_R}\big[\Delta u(y) \partial_{\nu(y)}  G(x,y) - G(x,y)\partial_\nu \Delta u(y)
\big]{\rm d}{s_{{y}}}+\int_{B_R}G(x,y)f(y){\rm d}{y}\nonumber\\
&=
\left\{\begin{array}{lll}
	u(x), &&x\in B_R,\\
	0,&& x\in\mathbb R^2\setminus\overline{B_R}.
\end{array}\right.
\end{align}

Noting that $\Phi_{\rm M}$ is a real-valued function, we let
\begin{align}\label{BifsS}
G^*(x,y)=-\frac{1}{2\kappa^2}\left(\Phi_{\rm H}^{*}(x,y)-\Phi_{\rm M}(x,y)\right),
\end{align}
where $\Phi_{\rm H}^{*}(x,y) =-\frac{{\rm i}}{4}H_{0}^{(2)}(\kappa|x-y|)$, with $H_0^{(2)}$ being the Hankel function of the second kind with order zero, satisfies
\begin{align}\label{sFS_P1}
&\Delta \Phi_{\rm H}^{*}(x,y)+\kappa^2 \Phi_{\rm H}^{*}(x,y)
=-\delta(x-y).
\end{align}
It can be verified from \eqref{BifsS}--\eqref{sFS_P1} that
\begin{align*}
(\Delta^2-\kappa^4)G^*(x,y)=-\delta(x-y). 
\end{align*}

By substituting $G$ with $G^*$ in \eqref{BiGTh-2} and then subtracting the resulting equation from \eqref{BiGTh-2}, we can deduce 
\begin{align}\label{2Gif}
\int_{B_R} \Psi(x,y) f(y){\rm d}{y}&=\int_{\partial B_R}\big[
\partial_{\nu(y)} \Psi(x, y)\Delta u(y)
- \Psi(x, y)\partial_\nu \Delta u(y)
\big]{\rm d}{s_{{y}}}
\nonumber\\
&\quad-\int_{\partial B_R}\big[
\Delta \Psi(x, y)\partial_\nu u(y) 
-\partial_{\nu(y)} \Delta \Psi(x, y)
u(y)\big]{\rm d}{s_{{y}}},
\end{align}
where 
\begin{align}\label{G-G*}
\Psi(x,y)=G(x,y)-G^*(x,y)=-\frac{{\rm i}}{4\kappa^2}J_{0}(\kappa|x-y|). 
\end{align}
Substituting \eqref{G-G*} into \eqref{2Gif}, we have from a straightforward calculation that 
\begin{align}\label{2Gif-1}
\int_{B_R}J_{0}(\kappa|x-y|)f(y){\rm d}{y}&=\int_{\partial B_R}\big[\partial_{\nu(y)}  J_{0}(\kappa|x-y|)\Delta u(y)
-J_{0}(\kappa|x-y|)\partial_\nu \Delta u(y)\big]{\rm d}{s_{y}}\nonumber\\
&\quad-\int_{\partial B_R}\big[ \Delta J_{0}(\kappa|x-y|)\partial_\nu u(y)-\partial_{\nu(y)} \Delta J_{0}(\kappa|x-y|)
u(y)\big]{\rm d}{s_{y}}\nonumber\\&
=-\int_{\partial B_R}\Big[\left(\nu\cdot \nabla_x J_{0}(\kappa|x-y|)\right) \Delta u(y)
+J_{0}(\kappa|x-y|)\partial_\nu \Delta u(y)\Big]{\rm d}{s_{y}}
\nonumber\\
&\quad-\int_{\partial B_R}\Big[\Delta J_{0}(\kappa|x-y|)\partial_\nu u(y)+\left(\nu\cdot\nabla_x \Delta J_{0}(\kappa|x-y|)\right)u(y)\Big]{\rm d}{s_{y}}.
\end{align}
Since $f$ is compactly supported in $B_R$, we let $J(x) = J_0(\kappa|x|)$ and obtain from taking the Fourier transform on both sides of \eqref{2Gif-1} that
\begin{equation}\label{2Gif-3}
 \hat{J}(\xi )\hat{f}(\xi )=\hat{J}(\xi ) \hat{U}(\xi ), 
\end{equation}
where the expression for $\hat U$ is given in terms of the near-field data of the wave field $u$ as follows:
\[
\hat{U}(\xi )=\int_{\partial B_R}\Big[\left(-({\rm i} \xi\cdot \nu)\Delta u(y)-\partial_\nu \Delta u(y)
+|\xi|^2\partial_\nu u(y)+| \xi|^2({\rm i} \xi\cdot \nu)u(y)\right)e^{-{\rm i} \xi \cdot y}\Big]{\rm d}{s_{y}}. 
\]

As evident from \eqref{2Gif-3}, it becomes apparent that $\hat f(\xi)=\hat U(\xi)$ when $\hat J(\xi)\neq 0$. In such cases, the source function $f$ can be reconstructed by performing the inverse Fourier transform on $\hat U(\xi)$. However, if $\hat J(\xi) = 0$, the information about $\hat f(\xi)$ is not present. As a result, attempting to determine $f$ by applying the inverse Fourier transform to $\hat U(\xi)$ becomes unfeasible.

Using \eqref{PWJAE-1}, we deduce through a straightforward calculation that 
\begin{align}\label{FTJ}
\hat{J}(\xi)
&=\int_{\mathbb{R}^2}J_0(\kappa|x|) e^{-{\rm i} \xi \cdot x}{\rm d}x
=\int_{\mathbb{R}^2}J_0(\kappa|x|) \left(\sum_{n=-\infty}^{+\infty}(-{\rm i})^{n}e^{{\rm i} n{\rm arg}(\hat{ \xi})}J_{n}(|{ \xi}||x|)e^{-{\rm i} n{\rm arg}(x)}\right){\rm d}x\nonumber\\
&=\int_0^\infty \int_0^{2\pi} J_0(\kappa r) \left(\sum_{n=-\infty}^{+\infty}(-{\rm i})^{n}e^{{\rm i} n{\rm arg}(\hat{ \xi})}J_{n}(|{ \xi}|r)e^{-{\rm i} n{\rm arg}(x)}\right)r{\rm d\theta}{\rm dr}\nonumber\\
&=2\pi\int_{0}^{\infty}J_0(\kappa r)J_{0}(|{ \xi}|r) r{\rm d}r\notag\\
&=\frac{2\pi}{\kappa}\delta(|{ \xi}|-\kappa)=\left\{\begin{array}{lll}
\infty,&& | \xi|=\kappa,\\
0, &&| \xi|\neq\kappa, 
\end{array}\right.
\end{align}
where $\delta$ is the Dirac delta function. Based on \eqref{FTJ}, it is evident that $\hat{J}(\xi)=0$ for $\xi\in\mathbb R^2$ except on the circle $| \xi|=\kappa$. By \eqref{2Gif-3}, we deduce that the information regarding $\hat{f}(\xi )$ can be determined only on the circle $|\xi|=\kappa$, in the following manner:
\begin{align}\label{fF}
\hat{f}(\xi )=\hat{U}(\xi ),\quad | \xi|=\kappa.
\end{align}

Furthermore, by applying Green's second theorem and using \eqref{bhwEq} and \eqref{Lp-de}, we derive
\begin{align*}
\check{f}( s)&
=\int_{B_R}f(y) e^{- s \cdot y}{\rm d}y
=-\int_{B_R}[(\Delta-\kappa^2)(\Delta+\kappa^2)u(y)] e^{- s \cdot y}{\rm d}y
\nonumber\\
&=-\int_{B_R}[(\Delta+\kappa^2)u(y)] [(\Delta-\kappa^2) e^{- s \cdot y}]{\rm d}y
\nonumber\\
&\quad-\int_{\partial B_R} \Big[e^{- s \cdot y}\partial_\nu\left(\Delta+\kappa^2)u(y)\right)-
\left((\Delta+\kappa^2)u(y)\right)
\partial_\nu e^{- s \cdot y}
\Big]{\rm d}{s_{y}}\nonumber\\
&=-\int_{B_R}[(\Delta+\kappa^2)u(y)] [(| s|^2-\kappa^2) e^{- s \cdot y}]{\rm d}y
\nonumber\\
&\quad-\int_{\partial B_R} \Big[\partial_\nu \Delta u(y)+\kappa^2 \partial_\nu u(y)+({ s \cdot \nu})\Delta u(y)+
\kappa^2({ s \cdot \nu})u(y)\Big]e^{- s \cdot y}{\rm d}{s_{y}}, 
\end{align*}
which implies, when considering the case $|s|=\kappa$, that 
\begin{equation}\label{Lp-near}
 \check{f}(s)=\check{V}(s),
\end{equation}
where $\hat V(s)$ is solely based on the near-field data of the function $u$ on the boundary $\partial B_R$:
\begin{align*}
\check{V}(s)=-\int_{\partial B_R} \Big[\partial_\nu \Delta u(y)+\kappa^2 \partial_\nu u(y)+({ s \cdot \nu})\Delta u(y)+
\kappa^2({ s \cdot \nu})u(y)\Big]e^{- s \cdot y}{\rm d}{s_{y}}.
\end{align*}

Hence, by employing Theorem \ref{2DThe-2}, in conjunction with \eqref{fF} and \eqref{Lp-near}, we are able to establish the following alternative characterization of nonradiating sources using the near-field data. 

\begin{theorem}\label{2DThe-4}
Assume that $f\in L^2(B_R)$. Then the source $f$ is nonradiating if and only if $\hat{U}(\xi )=\check{V}(s)=0$ for $|\xi|=|s|=\kappa$.
\end{theorem}

\subsection{Nonuniqueness}

In this section, we examine the topic of nonuniqueness of the inverse source problem. As the inverse source problem exhibits linearity, the presence of nonuniqueness indicates that a nonzero source function has the capability to produce a localized wave field. In other words, the nonuniqueness issue originates from the existence of nonradiating sources. We begin by introducing the null space, a concept directly related to the nonuniqueness aspects of the inverse source problem.

\begin{definition}
The source function $f\in L^2(B_R)$ is said to be in the null space $\mathcal{N}(R)$ if it satisfies
\begin{align}\label{Epeq}
\int_{B_R}J_{0}(\kappa|x-y|)f(y){\rm d}{y}
=0,\quad \int_{B_{R}}\Phi_{\rm M}(x,y)f(y){\rm d}y=0, \quad|x| > R. 
\end{align}
\end{definition}

For $x, y\in\mathbb{R}^2$, it is noteworthy that $J_{0}$ has a cylindrical harmonic expansion, expressed as
\begin{align}\label{JJFS}
J_{0}(\kappa|x-y|)
=\sum_{n=-\infty}^{\infty}J_{n}(\kappa|x|)e^{{\rm i} n {\rm arg}(x)}J_{n}(\kappa|y|)e^{-{\rm i} n {\rm arg}(y)},\quad |x|>|y|.
\end{align}
Hence, for $|x|>R$, we can deduce from \eqref{Epeq} and \eqref{JJFS} that
\begin{align*}
\int_{B_R}J_{0}(\kappa|x-y|)f(y){\rm d}{y}
&=\sum_{n=-\infty}^{\infty}J_{n}(\kappa|x|)e^{{\rm i} n {\rm arg}(x)}\int_{B_R}J_{n}(\kappa|y|)e^{-{\rm i} n {\rm arg}(y)}f(y){\rm d}{y}\nonumber\\
&=2\pi\bigg[\sum_{n=-\infty}^{\infty}J_{n}(\kappa|x|)e^{{\rm i} n {\rm arg}(x)}\int_{0}^{R}J_{n}(\kappa r)r\bigg(\frac{1}{2\pi}\int_{0}^{2\pi}f(y) e^{-{\rm i} n\theta}{\rm d} \theta\bigg){\rm d}r\bigg]\nonumber\\
&=2\pi\bigg[\sum_{n=-\infty}^{\infty}J_{n}(\kappa|x|)e^{{\rm i} n {\rm arg}(x)}\int_{0}^{R}f_n(r)J_{n}(\kappa r)r{\rm d}r\bigg], 
\end{align*}
which implies that
\begin{align}\label{JJFS-2}
\int_{B_R}J_{0}(\kappa|x-y|)f(y){\rm d}{y}
=2\pi\sum_{n=-\infty}^{\infty}\alpha_n J_{n}(\kappa|x|)e^{{\rm i} n {\rm arg}(x)},\quad |x| > R.
\end{align}

\begin{proposition}\label{2DThe-FHnull}
Assume that $f\in L^2(B_R)$. Then the source $f$ is nonradiating if and only if $f \in \mathcal{N}(R)$.
\end{proposition}

\begin{proof}
If $f\in L^2(B_R)$ is a nonradiating source for the biharmonic wave equation, according to Theorem \ref{2DThe-cn}, Proposition \ref{2DThe-fHct}, and \eqref{JJFS-2}, we can derive \eqref{Epeq}, which implies that $f$ belongs to $\mathcal{N}(R)$.

Conversely, if $f\in L^2(B_R)$ belongs to $\mathcal{N}(R)$, it satisfies \eqref{Epeq}. By invoking \eqref{JJFS-2} and utilizing Proposition \ref{2DThe-fHct}, we can conclude that $\alpha_n=\beta_n=0$ holds for all $\mathbb{Z}$, thereby confirming that $f$ is indeed a nonradiating source.
\end{proof}

Proposition \ref{2DThe-FHnull} demonstrates that the null space includes a set of nonradiating sources. In Section \ref{S:ex}, we establish the non-emptiness of $\mathcal{N}(R)$ by providing explicit examples of nonradiating sources. Consequently, the presence of nonradiating sources leads to the lack of uniqueness in the inverse source problem. 

\section{The Three-Dimensional problem}\label{S:3d}

This section is dedicated to examining the characterization of nonradiating sources within the context of the three-dimensional biharmonic wave equation. The analysis presented here is analogous to that of the two-dimensional case.

\subsection{Characterizations of nonradiating sources}

In three dimensions, the Green's functions for the Helmholtz and modified Helmholtz equations are given by 
\begin{align*}
\Phi_{\rm H}(x,y)=\frac{e^{{\rm i}\kappa|x-y|}}{4\pi|x-y|},\quad\quad
\Phi_{\rm M}(x,y)=\frac{e^{-\kappa|x-y|}}{4\pi|x-y|}. 
\end{align*}
By the addition theorem (cf. \cite[Theorem 2.11]{DR-2013}), they admit the following expansions for $|x|>|y|$ :
\begin{align}
\Phi_{\rm H}(x,y)
&={\rm i}\kappa\sum_{n=0}^{\infty}\sum_{m=-n}^{n}h_{n}^{(1)}(\kappa|x|)Y_{n}^{m}( {\hat x})j_{n}(\kappa|y|)\overline{Y_{n}^{m}( {\hat y})},\label{3DGraf-1}\\
\Phi_{\rm M}(x,y)
&=-\kappa\sum_{n=0}^{\infty}\sum_{m=-n}^{n}h_{n}^{(1)}({\rm i}\kappa|x|)Y_{n}^{m}( {\hat x})j_{n}({\rm i}\kappa|y|)\overline{Y_{n}^{m}( {\hat y})},\label{3DGraf-2}
\end{align}
where ${\hat x}=x/|x|, {\hat y}=y/|y|$,  the spherical harmonics $\{Y_{n}^{m}: m=-n, \dots, n, n=0, 1, \dots\}$ form a complete orthonormal system in the space of square integrable functions on the unit sphere, $j_n$ denotes spherical Bessel function of order $n$, and $h_{n}^{(1)}$ stands for the spherical Hankel function of the first kind with order $n$.

Let
\begin{align}\label{3Dfhm}
g_{\rm H}(x)=-\int_{B_{R}}\Phi_{\rm H}(x,y)f(y){\rm d}y,\quad\quad
g_{\rm M}(x)=-\int_{B_{R}}\Phi_{\rm M}(x,y)f(y){\rm d}y. 
\end{align}
We have from \eqref{BiFS}--\eqref{Bihs} and \eqref{3Dfhm} that
\begin{align}\label{3DBinse}
u(x)&=\int_{B_{R}}G(x,y)f(y){\rm d}y
=-\frac{1}{2\kappa^2}\left(\int_{B_{R}}\Phi_{\rm H}(x,y)f(y){\rm d}y-\int_{B_{R}}\Phi_{\rm M}(x,y)f(y){\rm d}y\right)\nonumber\\
&=\frac{1}{2\kappa^2}\left(g_{\rm H}(x)-g_{\rm M}(x)\right).
\end{align}

\begin{theorem}\label{3DThe-nrs}
The source $f$ is nonradiating if and only if $g_{\rm H}(x)= g_{\rm M}(x)=0$ for $|x|>R$.
\end{theorem}

\begin{proof}
It is obvious from \eqref{3DBinse} that $f$ is nonradiating if $g_{\rm H}(x)= g_{\rm M}(x)=0$ for $|x|>R$. Conversely, if $f$ is nonradiating, then from \eqref{3DBinse}, it is evident that $u(x)=0$ for $|x| > R$, and this dependence solely relies on $g_{\rm H}(x)$ and $g_{\rm M}(x)$. In the scenario where $g_{\rm H}(x)\neq 0$ and $g_{\rm M}(x)\neq 0$ for $|x| > R$, we can confirm, based on \eqref{2Dax-y}, that as $|x|\to \infty$, the following asymptotic expansions hold:
\begin{align}\label{3DGH}
\Phi_{\rm H}(x,y)
=\frac{e^{{\rm i}\kappa|x-y|}}{4\pi|x-y|}
=\frac{e^{{\rm i}\kappa|x|}}{4\pi|x|}\bigg[e^{-{\rm i}\kappa {\hat x} \cdot y}+O\bigg(\frac{1}{| {x}| }\bigg)\bigg]
\end{align}
and
\begin{align}\label{3DGM}
\Phi_{\rm M}(x,y)
=\frac{e^{-\kappa|x-y|}}{4\pi|x-y|}
=\frac{e^{{\rm i}\kappa|x|}}{4\pi|x|}\bigg[e^{-\kappa({\rm i}|x|+|x|- {\hat x} \cdot y)}\bigg(1+O\bigg(\frac{1}{| {x}| }\bigg)\bigg)\bigg]. 
\end{align}
Hence, from \eqref{3DGH} and \eqref{3DGM}, we have
\begin{align}\label{3DGab}
G(x,y)
=-\frac{1}{8\pi\kappa^2}\frac{e^{{\rm i}\kappa|x|}}{|x|}\bigg[e^{-{\rm i}\kappa {\hat x} \cdot y}+O\bigg(\frac{1}{| {x}| }\bigg)\bigg],\quad  |x|\to \infty. 
\end{align}
Substituting \eqref{3DGab} into \eqref{3DBinse} yields 
\begin{align*}
u(x)=-\frac{1}{8\pi\kappa^2}\frac{e^{{\rm i}\kappa|x|}}{|x|}\bigg[\int_{D}e^{-{\rm i}\kappa {\hat x} \cdot y}f(y){\rm d}y+O\bigg(\frac{1}{| {x}| }\bigg)\bigg],\quad  |x|\to \infty. 
\end{align*}
which implies that $u(x)\not\equiv 0$ for $|x| > R$. Therefore, ensuring $u(x)= 0$ for $|x| > R$ is equivalent to the condition $g_{\rm H}(x)= g_{\rm M}(x)= 0$ for $|x| > R$.
\end{proof}

According to Theorem \ref{3DThe-nrs}, a source $f(x)$ is deemed nonradiating for the biharmonic wave equation if and only if it satisfies the condition of being a nonradiating source for both the Helmholtz equation and the modified Helmholtz equation.

Combining \eqref{3DGraf-1}--\eqref{3DGraf-2} and \eqref{3DBinse}, we obtain 
\begin{align}\label{3DBihss}
u(x)&=\int_{B_{R}}G(x,y)f(y){\rm d}y
=\frac{1}{2\kappa^2}\left(g_{\rm H}(x)-g_{\rm M}(x)\right)\nonumber\\
&=-\frac{1}{2\kappa^2}\bigg[{\rm i}\kappa\sum_{n=0}^{\infty}\sum_{m=-n}^{n}h_{n}^{(1)}(\kappa|x|)Y_{n}^{m}( {\hat x})\int_{B_{R}}j_{n}(\kappa|y|)\overline{Y_{n}^{m}( {\hat y})}f(y){\rm d}y\nonumber\\
&\quad+\kappa\sum_{n=0}^{\infty}\sum_{m=-n}^{n}h_{n}^{(1)}({\rm i}\kappa|x|)Y_{n}^{m}( {\hat x})\int_{B_{R}}j_{n}({\rm i}\kappa|y|)\overline{Y_{n}^{m}( {\hat y})}f(y){\rm d}y\bigg]\nonumber\\
&=-\frac{1}{2\kappa^2}\bigg[{\rm i}\kappa\sum_{n=0}^{\infty}\sum_{m=-n}^{n}h_{n}^{(1)}(\kappa|x|)Y_{n}^{m}( {\hat x})
\int_{0}^{R}j_{n}(\kappa r)r^2\bigg(\int_{0}^{\pi}\int_{0}^{2\pi} f(r,\theta,\phi) \overline{Y_{n}^{m}(\theta,\phi)}\sin\theta{\rm d}\theta{\rm d}\phi\bigg){\rm d}r\bigg]\nonumber\\
&\quad+
\kappa\sum_{n=0}^{\infty}\sum_{m=-n}^{n}h_{n}^{(1)}({\rm i}\kappa|x|)Y_{n}^{m}( {\hat x})
\int_{0}^{R}j_{n}({\rm i}\kappa r)r^2\bigg(\int_{0}^{\pi}\int_{0}^{2\pi} f(r,\theta,\phi) \overline{Y_{n}^{m}(\theta,\phi)}\sin\theta{\rm d}\theta{\rm d}\phi\bigg){\rm d}r\bigg]\nonumber\\
&=-\frac{1}{2\kappa^2}\bigg[{\rm i}\kappa\sum_{n=0}^{\infty}\sum_{m=-n}^{n}h_{n}^{(1)}(\kappa|x|)Y_{n}^{m}( {\hat x})\int_{0}^{R}f_n^{m}(r)j_{n}(\kappa r)r^2{\rm d}r\nonumber\\
&\quad+\kappa\sum_{n=0}^{\infty}\sum_{m=-n}^{n}h_{n}^{(1)}({\rm i}\kappa|x|)Y_{n}^{m}( {\hat x})\int_{0}^{R}f_n^{m}(r)j_{n}({\rm i}\kappa r)r^2{\rm d}r\bigg]\nonumber\\
&=-\frac{1}{2\kappa^2}\bigg[{\rm i}\kappa\sum_{n=0}^{\infty}\sum_{m=-n}^{n}\alpha_n^m h_{n}^{(1)}(\kappa|x|)Y_{n}^{m}( {\hat x}) +\kappa\sum_{n=0}^{\infty}\sum_{m=-n}^{n}\beta_n^m h_{n}^{(1)}({\rm i}\kappa|x|)Y_{n}^{m}( {\hat x})\bigg],
\end{align}
where
\begin{align*}
f_n^{m}(r)=\int_{0}^{\pi}\int_{0}^{2\pi} f(r,\theta,\phi) \overline{Y_{n}^{m}(\theta,\phi)}\sin\theta{\rm d}\theta{\rm d}\phi,
\end{align*}
and
\begin{align*} 
\alpha_n^m=\int_{0}^{R}f_n^{m}(r)j_{n}(\kappa r)r^2{\rm d}r,\quad \quad
\beta_n^m=\int_{0}^{R}f_n^{m}(r)j_{n}({\rm i}\kappa r)r^2{\rm d}r.
\end{align*}
It is evident that $f_n^m$ also represents the Fourier coefficients of $f$ in the expansion
\begin{align*}
f(x)=\sum_{n=0}^{\infty}\sum_{m=-n}^{n}f_n^m(r)Y_n^m(\theta,\phi)
\end{align*}
with respect to the complete set of spherical harmonics $Y_n^m(\theta,\phi)$.

It can be seen from \eqref{3DBihss} that the wave field $u(x)$ for $|x| > R$ relies on the coefficients $\alpha_n^m$ and $\beta_n^m$, which represent the projections of the coefficients $f_n^m$ onto the functions $j_{n}(\kappa r)r^2$ and $j_{n}({\rm i}\kappa r)r^2$, respectively. As a result, it becomes possible to construct a source $f$ that does not yield any radiated field by nullifying these projections. By closely examining the expansions \eqref{3DBihss} of the solution and referring to Theorem \ref{3DThe-nrs}, we are able to derive the subsequent theorem.

\begin{theorem}\label{3DThe-ctc}
The source $f$ is nonradiating if and only if $\alpha_n^m=\beta_n^m=0$ for all $n = 0,1,\dots, m = -n, \dots, n.$
\end{theorem}

\begin{proposition}\label{3DThe-gHct}
Assume that $f\in L^2(B_R)$. Then $g_{\rm M}(x)|_{\partial B_R}=0$ if and only if $\beta_n^m=0$ for all $n = 0,1,\cdots, m = -n,\cdots, n$.
\end{proposition}

\begin{proof}
From \eqref{3DGraf-2}--\eqref{3Dfhm} and \eqref{3DGM}, for $f\in L^2(B_R)$, we know that $g_{\rm M}$ satisfies the modified Helmholtz equation
\begin{equation}\label{3D-mhe}
 \Delta g_{\rm M}-\kappa^2 g_{\rm M}=f\quad \text{in} ~ \mathbb R^3
\end{equation}
and
\begin{align}\label{ghct-1}
g_{\rm M}(x)
=-\int_{B_{R}}\Phi_{\rm M}(x,y)f(y){\rm d}y
=-\kappa\sum_{n=0}^{\infty}\sum_{m=-n}^{n}\beta_n^m h_{n}^{(1)}({\rm i}\kappa|x|)Y_{n}^{m}( {\hat x}),\quad |x|>R.
\end{align}
It follows from \eqref{3D-mhe} and the boundary condition $g_{\rm M}=0$ on $\partial B_R$ and the radiation condition, we have $g_{\rm M}(x)=0$ for $|x|>R$, which implies that $\beta_n^m=0$ for all $n = 0,1,\dots, m = -n,\cdots, n$.

From $\beta_n^m=0$ for all $n = 0,1,\dots, m = -n,\cdots, n$ and \eqref{ghct-1}, we have $g_{\rm M}(x)=0$ for $|x|>R$. Then, combining $g_{\rm M}(x)=0$ for $|x|>R$ and $ g_{\rm M}\in H^{\frac{3}{2}}(\partial B_R)$, we find that $g_{\rm M}(x)=0$ on $\partial B_R$. 
\end{proof}

Similarly, we may also show the following result. 

\begin{proposition}\label{3DThe-gMct}
Assume that $f\in L^2(B_R)$.  Then $g_{\rm H}(x)|_{\partial B_R}=0$ if and only if $\alpha_n^m=0$ for all $n = 0,1,\dots, m = -n,\dots, n$.
\end{proposition}

Combining Theorem \ref{3DThe-ctc}, Proposition \ref{3DThe-gHct} and \ref{3DThe-gMct}, we find that the source $f\in L^2(B_R)$ is nonradiating if and only if $g_{\rm H}(x)|_{\partial B_R}=g_{\rm M}(x)|_{\partial B_R}=0$.

Next we derive an alternative characterization of nonradiating sources in $\mathbb{R}^3$. Recall the Jacobi--Anger expansion for the plane wave (cf. \cite[Theorem 2.8 and $(2.46)$]{DR-2013}):
\begin{align}\label{3DPWJAE}
e^{{\rm i}\kappa{x}\cdot{ d}}&=4\pi\sum_{n=0}^{+\infty}
\sum_{m=-n}^{n}
{{\rm i}}^{n}j_{n}(\kappa|x|)Y_n^m({\hat x})\overline{Y_n^m({d})},\quad x\in\mathbb{R}^3,
\end{align}
where $d\in\mathbb R^3$ is a unit propagation direction vector. Then, for $ \xi\in\mathbb{R}^3$, we have

\begin{align}\label{3DPWJAE-1}
e^{-{\rm i}{ \xi}\cdot{x}}
=4\pi\sum_{n=0}^{+\infty}
\sum_{m=-n}^{n}
(-{{\rm i}})^{n}j_{n}(|{ \xi}||x|)\overline{Y_n^m({\hat x})}Y_n^m({\hat\xi}),\quad x\in\mathbb{R}^3,
\end{align}
where $\hat{ \xi}={ \xi}/|{ \xi}|$. Using \eqref{3DPWJAE-1}, we obtain 
\begin{align*}
\hat{f}(\xi)
&=\int_{B_{R}}f(x) e^{-{\rm i} \xi \cdot x}{\rm d}x
=\int_{B_{R}}f(x) \left(4\pi\sum_{n=0}^{+\infty}
\sum_{m=-n}^{n}
(-{{\rm i}})^{n}j_{n}(|{ \xi}||x|)\overline{Y_n^m({\hat x})}Y_n^m({\hat\xi})\right){\rm d}x
\nonumber\\
&=4\pi\sum_{n=0}^{+\infty}
\sum_{m=-n}^{n}(-{\rm i})^{n}Y_n^m({\hat\xi})
\int_{0}^{R}j_{n}(|{ \xi}|r)r^2\bigg(\int_{0}^{\pi}\int_{0}^{2\pi}f(r,\theta,\phi) \overline{Y_n^m(\theta,\phi)}\sin\theta{\rm d} \theta{\rm d}\phi\bigg){\rm d}r
\nonumber\\
&=4\pi\sum_{n=0}^{+\infty}
\sum_{m=-n}^{n}(-{\rm i})^{n}Y_n^m({\hat\xi})\int_{0}^{R}f_n^m(r)j_{n}(|{ \xi}|r)r^2{\rm d}r,
\end{align*}
which implies
\begin{align}\label{3Dxik}
\hat{f}(\xi)
=4\pi\sum_{n=0}^{+\infty}
\sum_{m=-n}^{n}(-{\rm i})^{n} \alpha_n^m Y_n^m({\hat\xi}),\quad |\xi|=\kappa.
\end{align}

From \eqref{3DPWJAE} and \cite[Proposition 3.1]{WZC-2021}, we have
\begin{align*}
e^{-\kappa{x}\cdot{ d}}
=e^{{\rm i}({\rm i}\kappa){x}\cdot{ d}}
=4\pi\sum_{n=0}^{+\infty}
\sum_{m=-n}^{n}
{{\rm i}}^{n}j_{n}({\rm i}\kappa|x|)Y_n^m({\hat x})\overline{Y_n^m({d})}. 
\end{align*}
Then, for $|x|=r\leq R$ and $0<| s|< \kappa+K$, $K>0$ is a bounded constant, we get
\begin{align*}
e^{-{ s}\cdot{x}}
=e^{-{|x| s}\cdot{ {\hat x}}}
=e^{-{r s}\cdot{ {\hat x}}}
=4\pi\sum_{n=0}^{+\infty}
\sum_{m=-n}^{n}
{{\rm i}}^{n}j_{n}({\rm i}|s|r)Y_n^m({\hat{s}})\overline{Y_n^m({\hat x})},
\end{align*}
where $\hat s=s/|s|$. 

Define
\begin{align}\label{3DLp-de}
\check{f}( s)&
=\int_{B_{R}}f(x) e^{- s \cdot x}{\rm d}x,\quad 0<| s|< \kappa+K. 
\end{align}
By  \eqref{3DLp-de}, we have
\begin{align}\label{3DLp-1}
\check{f}( s)&
=\int_{B_{R}}f(x) \left(4\pi\sum_{n=0}^{+\infty}
\sum_{m=-n}^{n}
{{\rm i}}^{n}j_{n}({\rm i}| s|r)Y_n^m({\hat{s}})\overline{Y_n^m({\hat x})}\right){\rm d}x
\nonumber\\
&=4\pi\sum_{n=0}^{+\infty}
\sum_{m=-n}^{n}{{\rm i}}^{n}Y_n^m({\hat s})
\int_{0}^{R}j_{n}({\rm i}|{ s}|r)r^2\bigg(\int_{0}^{\pi}\int_{0}^{2\pi}f(r,\theta,\phi) \overline{Y_n^m(\theta,\phi)}\sin\theta{\rm d} \theta{\rm d}\phi\bigg){\rm d}r
\nonumber\\
&=4\pi\sum_{n=0}^{+\infty}
\sum_{m=-n}^{n}{{\rm i}}^{n}Y_n^m({\hat s})\int_{0}^{R}f_n^m(r)j_{n}({\rm i}|{ s}|r)r^2{\rm d}r,\quad 0<| s|< \kappa+K,
\end{align}
which implies 
\begin{align}\label{3DLp-1-1}
\check{f}( s)
=4\pi\sum_{n=0}^{+\infty}
\sum_{m=-n}^{n}{{\rm i}}^{n}\beta_n^m Y_n^m({\hat s}),\quad
| s|=\kappa.
\end{align}

According to Theorem \ref{3DThe-ctc}, together with \eqref{3Dxik} and \eqref{3DLp-1-1}, we arrive at the following conclusion.

\begin{theorem}\label{3DThe-2}
Assume that $f\in L^2(B_R)$. Then the source $f$ is nonradiating if and only if $\hat{f}(\xi)=\check{f}( s)=0$ when $|\xi|=|s|=\kappa$. 
\end{theorem}

\subsection{Near-field data}

First, applying Green's second theorem, from \eqref{bhwEq}, we have
\begin{align}\label{3DFt-near}
\hat{f}(\xi)&=\int_{B_{R}}f(y) e^{-{\rm i} \xi \cdot y}{\rm d}y
=-\int_{B_R}[(\Delta+\kappa^2)(\Delta-\kappa^2)u(y)] e^{-{\rm i} \xi \cdot y}{\rm d}y
\nonumber\\
&=-\int_{B_R}[(\Delta-\kappa^2)u(y)] [(\Delta+\kappa^2) e^{-{\rm i} \xi \cdot y}]{\rm d}y
\nonumber\\
&\quad-\int_{\partial B_R} \big[e^{-{\rm i} \xi \cdot y}\partial_\nu\big((\Delta-\kappa^2)u(y)\big)-
\big((\Delta-\kappa^2)u(y)\big)\partial_\nu e^{-{\rm i} \xi \cdot y}
\big]{\rm d}{s_{y}}\nonumber\\
&=-\int_{B_R}[(\Delta-\kappa^2)u(y)] [(-| \xi|^2+\kappa^2) e^{-{\rm i} \xi \cdot y}]{\rm d}y
\nonumber\\
&\quad-\int_{\partial B_R} \big[\partial_\nu \Delta u(y)-\kappa^2\partial_\nu u(y)+({{\rm i} \xi \cdot \nu})\Delta u(y)-
\kappa^2({{\rm i} \xi \cdot \nu})u(y)
\big]e^{-{\rm i} \xi \cdot y}{\rm d}{s_{y}},\quad  \xi\in \mathbb{R}^3.
\end{align}
By \eqref{3DFt-near}, it is evident that we can only determine $\hat{f}(\xi )$ on the sphere $|\xi|=\kappa, \xi\in\mathbb R^3$ as
\begin{align}\label{3DfF}
\hat{f}(\xi )&=\hat{U}(\xi)\notag\\
:&=-\int_{\partial B_R} \big[\partial_\nu \Delta u(y)-\kappa^2\partial_\nu u(y)+({{\rm i} \xi \cdot \nu})\Delta u(y)-
\kappa^2({{\rm i} \xi \cdot \nu})u(y)
\big]e^{-{\rm i} \xi \cdot y}{\rm d}{s_{y}}.
\end{align}

Furthermore, applying Green's second theorem, from \eqref{bhwEq} and \eqref{3DLp-1}, we get
\begin{align*}
\check{f}( s)&
=\int_{B_R}f(y) e^{- s \cdot y}{\rm d}y
=-\int_{B_R}[(\Delta-\kappa^2)(\Delta+\kappa^2)u(y)] e^{- s \cdot y}{\rm d}y
\nonumber\\
&=-\int_{B_R}[(\Delta+\kappa^2)u(y)] [(\Delta-\kappa^2) e^{- s \cdot y}]{\rm d}y
\nonumber\\
&\quad-\int_{\partial B_R} \big[e^{- s \cdot y}\partial_\nu \big((\Delta+\kappa^2)u(y)\big)-
\big((\Delta+\kappa^2)u(y)\big)\partial_\nu e^{- s \cdot y}
\big]{\rm d}{s_{y}}\nonumber\\
&=-\int_{B_R}[(\Delta+\kappa^2)u(y)] [(| s|^2-\kappa^2) e^{- s \cdot y}]{\rm d}y
\nonumber\\
&\quad-\int_{\partial B_R} \big[\partial_\nu \Delta u(y)+\kappa^2\partial_\nu u(y)+({ s \cdot \nu})\Delta u(y)+
\kappa^2({ s \cdot \nu})u(y)
\big]e^{- s \cdot y}{\rm d}{s_{y}},\quad 0<| s|< \kappa+K.
\end{align*}
Hence, we can only determine $\check{f}( s)$ on $|s|=\kappa, s\in\mathbb R^3$ as
\begin{align}\label{3DLp-near}
\check{f}( s)
&=\check{V}( s )\nonumber\\
&:=-\int_{\partial B_R} \big[\partial_\nu \Delta u(y)+\kappa^2\partial_\nu u(y)+({ s \cdot \nu})\Delta u(y)+
\kappa^2({ s \cdot \nu})u(y)
\big]e^{- s \cdot y}{\rm d}{s_{y}}.
\end{align}

By Theorem \ref{3DThe-2}, \eqref{3DfF} and \eqref{3DLp-near}, we have the following Theorem. 

\begin{theorem}\label{3DThe-4}
Assume that $f\in L^2(B_R)$. Then the source $f$ is nonradiating if and only if $\hat{U}(\xi )=\check{V}(s)=0$ when $|\xi|=|s|=\kappa$. 
\end{theorem}

\subsection{Nonuniqueness}

For the three-dimensional case, we define the following null space. 

\begin{definition}
The source function $f\in L^2(B_R)$ is said to be in the null space $\mathcal{N}(R)$ if
\begin{align}\label{3DEpeq}
\int_{B_R}j_{0}(\kappa|x-y|)f(y){\rm d}{y}
=0,\quad
\int_{B_R}\Phi_{\rm M}(x,y)f(y){\rm d}{y}
=0,\quad |x| > R.
\end{align}
\end{definition}

For $x, y\in\mathbb{R}^3$, we note that $j_{0}$ has the spherical harmonic expansion
\begin{align}\label{3DJJFS}
j_{0}(\kappa|x-y|)
=4\pi\sum_{n=0}^{\infty}\sum_{m=-n}^{n}j_{n}(\kappa|x|)Y_{n}^{m}( {\hat x})j_{n}(\kappa|y|)\overline{Y_{n}^{m}( {\hat y})},\quad |x|>|y|.
\end{align}
Hence, we have from \eqref{3DEpeq} and \eqref{3DJJFS} that
\begin{align}\label{3DJJFS-1}
&\int_{B_R}j_{0}(\kappa|x-y|)f(y){\rm d}{y}
=4\pi\sum_{n=0}^{\infty}\sum_{m=-n}^{n}j_{n}(\kappa|x|)Y_{n}^{m}( {\hat x})\int_{B_R}j_{n}(\kappa|y|)\overline{Y_{n}^{m}( {\hat y})}f(y){\rm d}{y}\nonumber\\
&=4\pi\sum_{n=0}^{\infty}\sum_{m=-n}^{n}j_{n}(\kappa|x|)Y_{n}^{m}( {\hat x})\int_{0}^{R}f_n^m(r)j_{n}(\kappa r)r^2{\rm d}r\notag\\
&=4\pi\sum_{n=0}^{\infty}\sum_{m=-n}^{n}\alpha_n^m j_{n}(\kappa|x|)Y_{n}^{m}( {\hat x}).
\end{align}

By Theorem \ref{3DThe-nrs}, \ref{3DThe-ctc} and \eqref{3DJJFS-1}, we conclude the following Theorem.

\begin{theorem}\label{3DThe-5}
Assume that $f\in L^2(B_R)$. Then the source $f$ is nonradiating if and only if $f \in \mathcal{N}(R)$.
\end{theorem}

\begin{proof}
If $f\in L^2(B_R)$ is a nonradiating source for the biharmonic wave equation, then, as a consequence of Theorems \ref{3DThe-nrs} and \ref{3DThe-ctc}, along with the relation \eqref{3DJJFS-1}, it follows that \eqref{3DEpeq} can be derived, which in turn implies $f$ belongs to the set $\mathcal{N}(R)$.

Conversely, if $f\in L^2(B_R)$ is an element of $\mathcal{N}(R)$, it satisfies \eqref{3DEpeq}. Additionally, according to \eqref{3DBihss} and \eqref{3DJJFS-1}, the conditions $\alpha_n^m=\beta_n^m=0$ hold for all $n = 0,1,\cdots, m = -n,\cdots, n$. By virtue of Theorem \ref{3DThe-ctc}, it is established that $f$ qualifies as a nonradiating source.
\end{proof}

\section{Examples of nonradiating sources}\label{S:ex}

In this section, we explicitly construct several examples of nonradiating sources to illustrate that the null space $\mathcal{N}(R)$ is not empty.

\subsection{General nonradiating sources}

Consider a smooth function $u\in C_{0}^{\infty}(D)$, where $D$ represents a domain with compact support. By applying the operator $-(\Delta^2-\kappa^4)$ to the function $u$, resulting in the function $f$, we obtain
\begin{equation}\label{Ex1}
f(x)=-(\Delta^2-\kappa^4)u(x),\quad x\in\mathbb{R}^d.
\end{equation}
By Definition \ref{def}, it is evident that the function $f$ given in \eqref{Ex1} also belongs to the space $C_{0}^{\infty}(D)$ and constitutes a nonradiating source for the biharmonic wave equation. Below, we demonstrate through equivalent characterizations that the function $f$ is a nonradiating source.

For $u\in C_{0}^{\infty}(B_R)$, we have 
$\partial_\nu \big((\Delta+\kappa^2)u(y)\big)|_{\partial B_R}=\big((\Delta+\kappa^2)u(y)\big)|_{\partial B_R}=0$. It follows from the integration by parts that 
\begin{align*}
f_{\rm M}(x)
&=-\int_{B_{R}}\Phi_{\rm M}(x,y)f(y){\rm d}y
=\int_{B_{R}}\Phi_{\rm M}(x,y)\big((\Delta-\kappa^2)(\Delta+\kappa^2)u(y)\big){\rm d}y\nonumber\\
&=\int_{B_{R}}\big((\Delta-\kappa^2)\Phi_{\rm M}(x,y)\big)\big((\Delta+\kappa^2)u(y)\big){\rm d}y\nonumber\\
&\quad+\int_{\partial B_R}\big[\Phi_{\rm M}(x,y)\partial_\nu \big((\Delta+\kappa^2)u(y)\big)-\big((\Delta+\kappa^2)u(y)\big)\partial_{\nu(y)} \Phi_{\rm M}(x,y)\big]{\rm d} s_{y}\nonumber\\
&=\int_{B_{R}}\big((\Delta-\kappa^2)\Phi_{\rm M}(x,y)\big)\big((\Delta+\kappa^2)u(y)\big){\rm d}y=0,\quad |x|>R.
\end{align*}
Similarly, we may show that 
\begin{align*}
f_{\rm H}(x)
=-\int_{B_{R}}\Phi_{\rm H}(x,y)f(y){\rm d}y=0,\quad |x|>R.
\end{align*}
According to Theorem \ref{2DThe-nrs} or Theorem \ref{3DThe-nrs}, the source function $f$ is nonradiating. 

Subsequently, we validate Theorem \ref{2DThe-cn} or Theorem \ref{3DThe-ctc} by showing that $\alpha_n=\beta_n=0$ for all $n\in\mathbb{Z}$. For the sake of brevity and generality, we focus our attention on the case where $d=2$.

For all $n\in \mathbb{Z}$, from \eqref{Ex1}, we have from the integration by parts that 
\begin{align*}
 \alpha_n&=\frac{1}{2\pi}\int_{B_R}f(y)J_{n}(\kappa |y|) e^{- {\rm i} n {\rm arg}( {y})}{\rm d}y
 =-\frac{1}{2\pi}\int_{B_R}\big((\Delta+\kappa^2)(\Delta-\kappa^2)u(y)\big)J_{n}(\kappa |y|) e^{- {\rm i} n {\rm arg}( {y})}{\rm d}y
 \nonumber\\
 &=-\frac{1}{2\pi}\int_{B_R}\big((\Delta-\kappa^2)u(y)\big)\big[(\Delta+\kappa^2)\big(J_{n}(\kappa |y|) e^{- {\rm i} n {\rm arg}( {y})}\big)\big]{\rm d}y
 \nonumber\\
 &=-\frac{1}{2\pi}\int_{0}^{R}\int_{0}^{2\pi} \bigg[\frac{1}{r}\bigg(\frac{\rm d}{{\rm d} r}\bigg(r\frac{\rm d}{{\rm d} r}\bigg)+\frac{1}{r}\frac{{\rm d}^2}{{\rm d} \theta^2}-\kappa^2 r\bigg)u(r,\theta)\bigg]\nonumber\\
 &\hspace{3cm}\times\bigg[\frac{1}{r}\bigg(\frac{\rm d}{{\rm d} r}\bigg(r\frac{\rm d}{{\rm d} r}\bigg)+\frac{1}{r}\frac{{\rm d}^2}{{\rm d} \theta^2}+\kappa^2 r\bigg) \big(J_{n}(\kappa r)e^{- {\rm i} n \theta}\big)\bigg]r {\rm d}\theta{\rm d}r\nonumber\\
&=-\frac{1}{2\pi}\int_{0}^{R}\int_{0}^{2\pi} \bigg[\frac{1}{r}\bigg(\frac{\rm d}{{\rm d} r}\bigg(r\frac{\rm d}{{\rm d} r}\bigg)+\frac{1}{r}\frac{{\rm d}^2}{{\rm d} \theta^2}-\kappa^2 r\bigg)u(r,\theta)\bigg]\nonumber\\
 &\hspace{3cm}\times\bigg[\bigg(\frac{\rm d}{{\rm d} r}\bigg(r\frac{\rm d}{{\rm d} r}\bigg)+\kappa^2 r-\frac{n^2}{r}\bigg) J_{n}(\kappa r)\bigg]e^{- {\rm i} n \theta} {\rm d}\theta{\rm d}r=0,
\end{align*}
where we utilize the property
\[
\bigg[\frac{\rm d}{{\rm d} r}\left(r\frac{\rm d}{{\rm d} r}\right)+\kappa^2 r-\frac{n^2}{r}\bigg] J_{n}(\kappa r)=0\quad \forall\, n\in\mathbb Z. 
\]
We can similarly show that $\beta_n=0$ for all $n\in\mathbb Z$ by noting 
\[
\bigg[\frac{\rm d}{{\rm d} r}\bigg(r\frac{\rm d}{{\rm d} r}\bigg)-\kappa^2 r-\frac{n^2}{r}\bigg] J_{n}({\rm i}\kappa r)=0. 
\]

\subsection{A two-dimensional nonradiating source}

For $\kappa > 0$ and $R > 0$, we consider the condition $J_0(\kappa R) = 0$, indicating that $\kappa R$ is a root of the Bessel function of order zero. Define
\begin{equation}\label{FM-1}
f_{\rm M}(x)=\left\{
\begin{aligned}
&\frac{J_0^3(\kappa |x|)}{\int_0^RJ_0^4(\kappa r) r{\rm d} r}-\frac{J_0^2(\kappa |x|)}{\int_0^RJ_0^3(\kappa r) r{\rm d} r}, &\quad |x|<R,\\
&0, &\quad |x|\geq R.
\end{aligned}
\right.
\end{equation}
It is clear to note $f_{\rm M}(x)|_{\partial B_R}=0$. Based on the identity $J_0'(z)=-J_1(z)$, a straightforward calculation gives
\begin{align*}
\partial_\nu f_{\rm M}(y)|_{\partial B_R}&=\lim_{h\to 0^+}\nu(y)\cdot\nabla f_{\rm M}(y-h\nu(y)) 
=\frac{-3\kappa J_0^2(\kappa R)J_{1}(\kappa R)}{\int_0^RJ_0^4(\kappa r) r{\rm d} r}+\frac{2\kappa J_0(\kappa R)J_{1}(\kappa R)}{\int_0^RJ_0^3(\kappa r) r{\rm d} r}=0,
\end{align*}
which indicates that $f_{\rm M}(x)\in C_0^1(\mathbb R^2)$. 

Consider the expression 
\begin{equation}\label{Ex2}
f(x)=(\Delta-\kappa^2)f_{\rm M}(x), 
\end{equation}
which is shown to be a nonradiating source for the two-dimensional biharmonic wave equation. 

It follows from \eqref{FM-1} and \eqref{Ex2} that 
\begin{align}\label{FM-4}
f(x) &=(\Delta-\kappa^2)f_{\rm M}(x)=
\left\{
\begin{aligned}
&\frac{(\Delta-\kappa^2)J_0^3(\kappa |x|)}{\int_0^RJ_0^4(\kappa r) r{\rm d} r}
-\frac{(\Delta-\kappa^2)J_0^2(\kappa |x|)}{\int_0^RJ_0^3(\kappa r) r{\rm d} r}, &\quad |x|<R,\\
&0, &\quad |x|\geq R.
\end{aligned}
\right.\nonumber\\
&=
\left\{
\begin{aligned}
&\frac{\frac{1}{r}\big(\frac{\rm d}{{\rm d} r}\big(r\frac{\rm d}{{\rm d} r}\big)-\kappa^2 r\big)J_0^3(\kappa r)}{\int_0^RJ_0^4(\kappa r) r{\rm d} r}
-\frac{\frac{1}{r}\big(\frac{\rm d}{{\rm d} r}\big(r\frac{\rm d}{{\rm d} r}\big)-\kappa^2 r\big)J_0^2(\kappa r)}{\int_0^RJ_0^3(\kappa r) r{\rm d} r}, &\quad |x|<R,\\
&0, &\quad |x|\geq R.
\end{aligned}
\right.
\end{align}
Using the integration by parts and the boundary conditions $f_{\rm M}(x)|_{\partial B_R}=\partial_\nu f_{\rm M}(x)|_{\partial B_R}=0$, we have
\begin{align}\label{FM-5}
&\int_{B_{R}}\Phi_{\rm H}(x,y)f(y){\rm d}y
=\int_{B_{R}}\Phi_{\rm H}(x,y)\left((\Delta-\kappa^2)f_{\rm M}(y)\right){\rm d}y\nonumber\\
&=\int_{B_{R}}\big((\Delta-\kappa^2)\Phi_{\rm H}(x,y)\big)f_{\rm M}(y){\rm d}y
+\int_{\partial B_R}\big(\Phi_{\rm H}(x,y)\partial_\nu f_{\rm M}(y)-f_{\rm M}(y)\partial_{\nu(y)} \Phi_{\rm H}(x,y)\big){\rm d} s_{y}\nonumber\\
&=\int_{B_{R}}\big((\Delta-\kappa^2)\Phi_{\rm H}(x,y)\big)f_{\rm M}(y){\rm d}y
=-2\kappa^2\int_{B_{R}}\Phi_{\rm H}(x,y)f_{\rm M}(y){\rm d}y,\quad |x|>R. 
\end{align}
On the other hand, we have from a straightforward calculation that 
\begin{align}\label{FM-6}
&\int_{B_{R}}\Phi_{\rm H}(x,y)f_{\rm M}(y){\rm d}y
=\int_{B_{R}}\frac{{\rm i}}{4}H_{0}^{(1)}(\kappa|x-y|)\left(\frac{J_0^3(\kappa |y|)}{\int_0^RJ_0^4(\kappa r) r{\rm d} r}-\frac{J_0^2(\kappa |y|)}{\int_0^RJ_0^3(\kappa r) r{\rm d} r}\right){\rm d}y\nonumber\\
&=\frac{{\rm i}}{4}\sum_{n=-\infty}^{\infty}H_{n}^{(1)}(\kappa|x|)e^{{\rm i} n {\rm arg}(x)}\int_{B_{R}} \left(\frac{J_0^3(\kappa |y|)}{\int_0^RJ_0^4(\kappa r) r{\rm d} r}-\frac{J_0^2(\kappa |y|)}{\int_0^RJ_0^3(\kappa r) r{\rm d} r}\right) J_{n}(\kappa |y|)e^{-{\rm i} n {\rm arg}(y)}{\rm d}y\nonumber\\
&=\frac{{\rm i}}{4}H_{0}^{(1)}(\kappa|x|) \left(\frac{2\pi\int_{0}^{R}J_0^4(\kappa r)r{\rm d}r}{\int_0^R J_0^4(\kappa r) r{\rm d} r}-\frac{2\pi\int_{0}^{R}J_0^3(\kappa r)r{\rm d}r}{\int_0^RJ_0^3(\kappa r) r{\rm d} r}\right)=0,\quad |x|>R.
\end{align}
Combining \eqref{FM-5} and \eqref{FM-6} yields 
\begin{align}\label{FM-07}
 f_{\rm H}(x)=-\int_{B_{R}}\Phi_{\rm H}(x,y)f(y){\rm d}y=0,\quad |x|>R. 
\end{align}

Similarly, we obtain from the integration by parts and \eqref{Phi_HM} that 
\begin{align}\label{FM-7}
f_{\rm M}(x)&=-\int_{B_{R}}\Phi_{\rm M}(x,y)f(y){\rm d}y
=-\int_{B_{R}}\Phi_{\rm M}(x,y)\big((\Delta-\kappa^2)f_{\rm M}(y)\big){\rm d}y\nonumber\\
&=-\int_{B_{R}}\big((\Delta-\kappa^2)\Phi_{\rm M}(x,y)\big)f_{\rm M}(y){\rm d}y\nonumber\\
&\quad-\int_{\partial B_R}\big(\Phi_{\rm M}(x,y)\partial_\nu f_{\rm M}(y)-f_{\rm M}(y)\partial_{\nu(y)} \Phi_{\rm M}(x,y)\big){\rm d} s_{y}\nonumber\\
&=-\int_{B_{R}}[(\Delta-\kappa^2)\Phi_{\rm M}(x,y)]f_{\rm M}(y){\rm d}y=0,\quad |x|>R.
\end{align}
Combining \eqref{FM-07} and \eqref{FM-7} leads to 
\begin{align*}
u(x)=\int_{B_R}G(x,y)f(y) dy=\frac{1}{2\kappa^2}(f_{\rm H}(x)-f_{\rm M}(x))=0,\quad |x|>R.
\end{align*}
which implies that $f(x)$ is a nonradiating source for the two-dimensional biharmonic wave equation.

Next, we verify that $\alpha_n=\beta_n=0$ for all $n\in \mathbb{Z}$. Noting 
\begin{align*}
J_0^k(\kappa r)\frac{\rm d}{{\rm d} r}\bigg(r\frac{\rm d}{{\rm d} r}J_0(\kappa r)\bigg)+\kappa^2 r J_0^{k+1}(\kappa r)&=\bigg[\frac{\rm d}{{\rm d} r}\bigg(r\frac{\rm d}{{\rm d} r}J_0(\kappa r)\bigg)+\kappa^2 r J_0(\kappa r)\bigg]J_0^k(\kappa r) \nonumber\\
&=0\times J_0^k (\kappa r)=0,
\end{align*}
we have for $k\geq 2$ that 
\begin{align*}
&-\int_{0}^{R}\big[\kappa^2 r J_0^k (\kappa r)\big]J_{0}(\kappa r){\rm d}r
=\int_{0}^{R}\bigg[\bigg(\frac{\rm d}{{\rm d} r}\bigg(r\frac{\rm d}{{\rm d} r}J_0(\kappa r)\bigg)\bigg]J_{0}^k(\kappa r){\rm d}r\nonumber\\
&=\bigg(r\frac{\rm d}{{\rm d} r}J_0(\kappa r)\bigg)J_{0}^k(\kappa r)\bigg|^{R}_{0}
-\int_{0}^{R}\bigg(r\frac{\rm d}{{\rm d} r}J_0(\kappa r)\bigg)\bigg(\frac{\rm d}{{\rm d} r}J_{0}^k(\kappa r)\bigg){\rm d}r\nonumber\\
&=-\int_{0}^{R}\bigg(\frac{\rm d}{{\rm d} r}J_0(\kappa r)\bigg)\bigg(r\frac{\rm d}{{\rm d} r}J_{0}^k(\kappa r)\bigg){\rm d}r\nonumber\\
&=-J_0(\kappa r)\bigg(r\frac{\rm d}{{\rm d} r}J_{0}^k(\kappa r)\bigg)\bigg|^{R}_{0}+\int_{0}^{R}J_0(\kappa r)\bigg[\frac{\rm d}{{\rm d} r}\bigg(r\frac{\rm d}{{\rm d} r}J_{0}^k(\kappa r)\bigg)\bigg]{\rm d}r\nonumber\\
&=-J_0(\kappa r)\bigg(rk J_{0}^{k-1}(\kappa r)\frac{\rm d}{{\rm d} r}J_{0}(\kappa r)\bigg)\bigg|^{R}_{0}+\int_{0}^{R}J_0(\kappa r)\bigg[\frac{\rm d}{{\rm d} r}\bigg(r\frac{\rm d}{{\rm d} r}J_{0}^k(\kappa r)\bigg)\bigg]{\rm d}r\nonumber\\
&=\int_{0}^{R}\bigg[\frac{\rm d}{{\rm d} r}\bigg(r\frac{\rm d}{{\rm d} r}J_{0}^k(\kappa r)\bigg)\bigg]J_0(\kappa r){\rm d}r,
\end{align*}
which implies 
\begin{align}\label{bess-2}
\int_{0}^{R}\bigg[\bigg(\frac{\rm d}{{\rm d} r}\bigg(r\frac{\rm d}{{\rm d} r}\bigg)-\kappa^2 r\bigg)J_0^k(\kappa r)\bigg]J_{0}(\kappa r){\rm d}r=-2\kappa^2\int_{0}^{R}  J_0^{k+1}(\kappa r)r{\rm d}r,\quad k\geq 2.
\end{align}

Hence, for all $n\in \mathbb{Z}$, from \eqref{bess-2} and \eqref{abn}, we get
\begin{align*}
\alpha_n&=\frac{1}{2\pi}\int_{B_{R}} \bigg[\frac{\frac{1}{r}\big[\frac{\rm d}{{\rm d} r}\big(r\frac{\rm d}{{\rm d} r}\big)-\kappa^2 r\big]J_0^3(\kappa r)}{\int_0^RJ_0^4(\kappa r) r{\rm d} r}
-\frac{\frac{1}{r}\big[\frac{\rm d}{{\rm d} r}\big(r\frac{\rm d}{{\rm d} r}\big)-\kappa^2 r\big]J_0^2(\kappa r)}{\int_0^R J_0^3(\kappa r) r{\rm d} r}\bigg]J_{n}(\kappa |y|)e^{-{\rm i} n {\rm arg}(y)}{\rm d} y\nonumber\\
&=\frac{1}{2\pi} \Bigg(\frac{\int_{0}^{R}\frac{1}{r}\big[\big(\frac{\rm d}{{\rm d} r}\big(r\frac{\rm d}{{\rm d} r}\big)-\kappa^2 r\big)J_0^3(\kappa r)\big]J_{n}(\kappa r)r{\rm d}r\int_{0}^{2\pi}e^{-{\rm i} n \theta }{\rm d}\theta}{\int_0^RJ_0^4(\kappa r) r{\rm d} r}\nonumber\\
&\hspace{3cm}
-\frac{\int_{0}^{R}\frac{1}{r}\big[\big(\frac{\rm d}{{\rm d} r}\big(r\frac{\rm d}{{\rm d} r}\big)-\kappa^2 r\big)J_0^2(\kappa r)\big]J_{n}(\kappa r)r{\rm d}r\int_{0}^{2\pi}e^{-{\rm i} n \theta }{\rm d}\theta}{\int_0^RJ_0^3(\kappa r) r{\rm d} r}\Bigg),
\end{align*}
which shows 
\begin{align*}
\alpha_n
&=\frac{1}{2\pi} \Bigg(\frac{\int_{0}^{R}\frac{1}{r}\big[\big(\frac{\rm d}{{\rm d} r}\big(r\frac{\rm d}{{\rm d} r}\big)-\kappa^2 r\big)J_0^3(\kappa r)\big]J_{n}(\kappa r)r{\rm d}r\times 0}{\int_0^RJ_0^4(\kappa r) r{\rm d} r}
\nonumber\\
&\hspace{3cm}
-\frac{\int_{0}^{R}\frac{1}{r}\big[\big(\frac{\rm d}{{\rm d} r}\big(r\frac{\rm d}{{\rm d} r}\big)-\kappa^2 r\big)J_0^2(\kappa r)\big]J_{n}(\kappa r)r{\rm d}r\times 0}{\int_0^RJ_0^3(\kappa r) r{\rm d} r}\Bigg) =0,\quad n\neq 0
\end{align*}
and
\begin{align*}
\alpha_0
&=\frac{1}{2\pi} \Bigg(\frac{2\pi\int_{0}^{R}\big[\big(\frac{\rm d}{{\rm d} r}\big(r\frac{\rm d}{{\rm d} r}\big)-\kappa^2 r\big)J_0^3(\kappa r)\big]J_{0}(\kappa r){\rm d}r}{\int_0^RJ_0^4(\kappa r) r{\rm d} r}\\
&\hspace{2cm}-\frac{2\pi\int_{0}^{R}\big[\big(\frac{\rm d}{{\rm d} r}\big(r\frac{\rm d}{{\rm d} r}\big)-\kappa^2 r\big)J_0^2(\kappa r)\big]J_{0}(\kappa r){\rm d}r}{\int_0^RJ_0^3(\kappa r) r{\rm d} r}\Bigg) \nonumber\\
&= \Bigg(\frac{(-2\kappa^2)\int_{0}^{R}  J_0^4(\kappa r)r{\rm d}r}{\int_0^RJ_0^4(\kappa r) r{\rm d} r}
-\frac{(-2\kappa^2)\int_{0}^{R}  J_0^3(\kappa r)r{\rm d}r}{\int_0^RJ_0^3(\kappa r) r{\rm d} r}\Bigg)=0.
\end{align*}

Noting 
\begin{align}\label{bess-4}
\bigg[\frac{\rm d}{{\rm d} r}\bigg(r\frac{\rm d}{{\rm d} r}J_0({\rm i}\kappa r)\bigg)-\kappa^2 r J_0({\rm i}\kappa r)\bigg]J_0^k(\kappa r)=0\times J_0^k(\kappa r)=0,\quad k\geq 2,
\end{align}
we can similarly show that 
\begin{align}\label{bess-7}
&\int_{0}^{R}\big(\kappa^2 r J_0({\rm i}\kappa r)\big)J^k_{0}(\kappa r){\rm d}r
=\int_{0}^{R}\bigg[J_{0}({\rm i}\kappa r)\bigg(\frac{\rm d}{{\rm d} r}\bigg(r\frac{\rm d}{{\rm d} r}J^k_{0}(\kappa r)\bigg)\bigg)\bigg]{\rm d}r. 
\end{align}
Combining \eqref{bess-4} and \eqref{bess-7} yields 
\begin{align}\label{bess-8}
&\int_{0}^{R}\bigg[\bigg(\frac{\rm d}{{\rm d} r}\bigg(r\frac{\rm d}{{\rm d} r}\bigg)-\kappa^2 r\bigg)J_0^k(\kappa r)\bigg]J_{0}({\rm i}\kappa r){\rm d}r\nonumber\\
&=
\int_{0}^{R}\bigg[\frac{\rm d}{{\rm d} r}\bigg(r\frac{\rm d}{{\rm d} r}J_0({\rm i}\kappa r)\bigg)-\kappa^2 r J_0({\rm i}\kappa r)\bigg]J_0^k(\kappa r){\rm d}r=0,\quad k\geq 2.
\end{align}
Then, for all $n\in \mathbb{Z}$, we obtain from \eqref{bess-8} and \eqref{abn} that 
\begin{align*}
\beta_n&=\frac{1}{2\pi}\int_{B_{R}} \bigg[\frac{\frac{1}{r}\big[\frac{\rm d}{{\rm d} r}\big(r\frac{\rm d}{{\rm d} r}\big)-\kappa^2 r\big]J_0^3(\kappa r)}{\int_0^RJ_0^4(\kappa r) r{\rm d} r}
-\frac{\frac{1}{r}\big[\frac{\rm d}{{\rm d} r}\big(r\frac{\rm d}{{\rm d} r}\big)-\kappa^2 r\big]J_0^2(\kappa r)}{\int_0^RJ_0^3(\kappa r) r{\rm d} r}\bigg]J_{n}({\rm i}\kappa |y|)e^{-{\rm i} n {\rm arg}(y)}{\rm d}y\nonumber\\
&=\frac{1}{2\pi} \Bigg(\frac{\int_{0}^{R}\frac{1}{r}\big[\big(\frac{\rm d}{{\rm d} r}\big(r\frac{\rm d}{{\rm d} r}\big)-\kappa^2 r\big)J_0^3(\kappa r)\big]J_{n}({\rm i}\kappa r)r{\rm d}r\int_{0}^{2\pi}e^{-{\rm i} n \theta }{\rm d}\theta}{\int_0^RJ_0^4(\kappa r) r{\rm d} r}\nonumber\\
&\hspace{3cm}-\frac{\int_{0}^{R}\frac{1}{r}\big[\big(\frac{\rm d}{{\rm d} r}\big(r\frac{\rm d}{{\rm d} r}\big)-\kappa^2 r\big)J_0^2(\kappa r)\big]J_{n}({\rm i}\kappa r)r{\rm d}r\int_{0}^{2\pi}e^{-{\rm i}n \theta }{\rm d}\theta}{\int_0^RJ_0^3(\kappa r) r{\rm d} r}\Bigg) \nonumber\\
&=\frac{1}{2\pi} \Bigg(\frac{\int_{0}^{R}\big[\big(\frac{\rm d}{{\rm d} r}\big(r\frac{\rm d}{{\rm d} r}\big)-\kappa^2 r\big)J_0^3(\kappa r)\big]J_{n}({\rm i}\kappa r)r{\rm d}r\int_{0}^{2\pi}e^{-{\rm i} n \theta }{\rm d}\theta}{\int_0^RJ_0^4(\kappa r) r{\rm d} r}\nonumber\\
&\hspace{3cm}-\frac{\int_{0}^{R}\big[\big(\frac{\rm d}{{\rm d} r}\big(r\frac{\rm d}{{\rm d} r}\big)-\kappa^2 r\big)J_0^2(\kappa r)\big]J_{n}({\rm i}\kappa r)r{\rm d}\int_{0}^{2\pi}e^{-{\rm i}n \theta }{\rm d}\theta}{\int_0^RJ_0^3(\kappa r) r{\rm d} r}\Bigg) \nonumber\\
&=\frac{1}{2\pi} \Bigg(\frac{0}{\int_0^RJ_0^4(\kappa r) r{\rm d} r}-\frac{0}{\int_0^RJ_0^3(\kappa r) r{\rm d} r}\Bigg) =0.
\end{align*}

\subsection{A three-dimensional nonradiating source}

Analogously, we can employ a similar approach to construct a nonradiating source for the three-dimensional biharmonic wave equation.

For $\kappa>0, R>0$, let $j_0(\kappa R)=0$ and
\begin{align}\label{MHFs}
g_{\rm H}(x)=\left\{
\begin{aligned}
&\frac{j_0^{m_1}(\kappa |x|)}{\int_0^R j_0^{m_1}(\kappa r) j_0({\rm i}\kappa r)r^2{\rm d} r}-\frac{j_0^{m_2}(\kappa |x|)}{\int_0^R j_0^{m_2}(\kappa r)j_0({\rm i}\kappa r) r^2{\rm d} r}, &\quad |x|<R,\\
&0, &\quad |x|\geq R, 
\end{aligned}
\right.
\end{align}
where $m_1, m_2$ are distinct positive integers greater than 2, i.e., $m_1\neq m_2$. Obviously, we have $g_{\rm H}(x)=\partial_\nu g_{\rm H}(x)=0$ on $\partial B_R$.
Let 
\begin{align}\label{Ex3}
f(x)&=(\Delta+\kappa^2)g_{\rm H}(x)
=\left\{
\begin{aligned}
&\frac{(\Delta+\kappa^2)j_0^{m_1}(\kappa |x|)}{\int_0^R j_0^{m_1}(\kappa r) j_0({\rm i}\kappa r)r^2{\rm d} r}-\frac{(\Delta+\kappa^2)j_0^{m_2}(\kappa |x|)}{\int_0^R j_0^{m_2}(\kappa r)j_0({\rm i}\kappa r) r^2{\rm d} r}, &\quad |x|<R,\\
&0, &\quad |x|\geq R, 
\end{aligned}
\right.\nonumber\\
&=\left\{
\begin{aligned}
&\frac{\frac{1}{r^2}\big(\frac{\rm d}{{\rm d} r}\big(r^2\frac{\rm d}{{\rm d} r}\big)
+\kappa^2 r^2\big)j_0^{m_1}(\kappa r)}{\int_0^R j_0^{m_1}(\kappa r) j_0({\rm i}\kappa r)r^2{\rm d} r}-\frac{\frac{1}{r^2}\big(\frac{\rm d}{{\rm d} r}\big(r^2\frac{d}{d r}\big)
+\kappa^2 r^2\big)j_0^{m_2}(\kappa r)}{\int_0^R j_0^{m_2}(\kappa r)j_0({\rm i}\kappa r) r^2{\rm d} r}, &\quad |x|<R,\\
&0, &\quad |x|\geq R. 
\end{aligned}
\right.
\end{align}
Subsequently, we proceed to demonstrate that the function defined in \eqref{Ex3} is indeed a nonradiating source for the three-dimensional biharmonic wave equation.

By the boundary conditions of $g_{\rm H}$ on $\partial B$ and integration by parts, we have
\begin{align*}
&\int_{B_{R}}\Phi_{\rm M}(x,y)f(y){\rm d}y
=\int_{B_{R}}\Phi_{\rm M}(x,y)\big((\Delta+\kappa^2)g_{\rm H}(y)\big){\rm d}y\nonumber\\
&=\int_{B_{R}}\big((\Delta+\kappa^2)\Phi_{\rm M}(x,y)\big)g_{\rm H}(y){\rm d}y
+\int_{\partial B_R}\big(\Phi_{\rm M}(x,y)\partial_\nu g_{\rm H}(y)-g_{\rm H}(y)\partial_{\nu(y)} \Phi_{\rm M}(x,y)\big){\rm d} s_{y}\nonumber\\
&=\int_{B_{R}}\big((\Delta+\kappa^2)\Phi_{\rm M}(x,y)\big)g_{\rm H}(y){\rm d}y
=2\kappa^2\int_{B_{R}}\Phi_{\rm M}(x,y)g_{\rm H}(y){\rm d}y,\quad |x|>R. 
\end{align*}
A straightforward calculation gives 
\begin{align*}
&\int_{B_{R}}\Phi_{\rm M}(x,y)g_{\rm H}(y){\rm d}y\nonumber\\
&=\int_{B_{R}}\frac{e^{-\kappa|x-y|}}{4\pi|x-y|} \left(\frac{j_0^{m_1}(\kappa |y|)}{\int_0^R j_0^{m_1}(\kappa r) j_0({\rm i}\kappa r)r^2{\rm d} r}-\frac{j_0^{m_2}(\kappa |y|)}{\int_0^R j_0^{m_2}(\kappa r)j_0({\rm i}\kappa r) r^2{\rm d} r}\right){\rm d}y\nonumber\\
&=-\kappa\sum_{n=0}^{\infty}\sum_{m=-n}^{n}h_{n}^{(1)}({\rm i}\kappa|x|)Y_{n}^{m}( {\hat x})\notag\\
&\hspace{2cm}\times \int_{B_{R}}\left(\frac{j_0^{m_1}(\kappa |y|)}{\int_0^R j_0^{m_1}(\kappa r) j_0({\rm i}\kappa r)r^2{\rm d} r}-\frac{j_0^{m_2}(\kappa |y|)}{\int_0^R j_0^{m_2}(\kappa r)j_0({\rm i}\kappa r) r^2{\rm d} r}\right) j_{n}({\rm i}\kappa|y|)\overline{Y_{n}^{m}( {\hat y})}{\rm d}y\nonumber\\
&=-\kappa\sum_{n=0}^{\infty}\sum_{m=-n}^{n}h_{n}^{(1)}({\rm i}\kappa|x|)Y_{n}^{m}( {\hat x}) \Bigg(\frac{\int_{0}^{R} j_0^{m_1}(\kappa r)j_{n}({\rm i}\kappa r)r^2{\rm d}r\big(\int_{0}^{\pi}\int_{0}^{2\pi}\overline{Y_{n}^{m}(\theta,\phi)}\sin\theta{\rm d} \theta{\rm d} \phi\big)}{\int_0^R j_0^{m_1}(\kappa r) j_0({\rm i}\kappa r)r^2{\rm d} r}
\nonumber\\
&\hspace{2cm}-\frac{\int_{0}^{R} j_0^{m_2}(\kappa r)j_{n}({\rm i}\kappa r)r^2{\rm d}r\big(\int_{0}^{\pi}\int_{0}^{2\pi}\overline{Y_{n}^{m}(\theta,\phi)}\sin\theta{\rm d} \theta{\rm d} \phi\big)}{\int_0^R j_0^{m_2}(\kappa r)j_0({\rm i}\kappa r) r^2{\rm d} r}\Bigg) \nonumber\\
&=-\kappa h_{0}^{(1)}({\rm i}\kappa|x|)Y_{0}^{0}( {\hat x})2\sqrt{\pi} \left(\frac{\int_{0}^{R} j_0^{m_1}(\kappa r)j_{0}({\rm i}\kappa r)r^2{\rm d}r}{\int_0^R j_0^{m_1}(\kappa r)j_0({\rm i}\kappa r) r^2{\rm d} r}-\frac{\int_{0}^{R} j_0^{m_2}(\kappa r)j_{0}({\rm i}\kappa r)r^2{\rm d}r}{\int_0^R j_0^{m_2}(\kappa r)j_0({\rm i}\kappa r) r^2{\rm d} r}\right)=0,\quad |x|>R.
\end{align*}
Then, we obtain 
\begin{align}\label{3DGM_ex}
g_{\rm M}(x)=-\int_{B_{R}}\Phi_{\rm M}(x,y)f(y){\rm d}y=0,\quad |x|>R. 
\end{align}

Similarly, we can obtain from the integration by parts that 
\begin{align}\label{3DFM-7}
g_{\rm H}(x)&=-\int_{B_{R}}\Phi_{\rm H}(x,y)f(y){\rm d}y
=-\int_{B_{R}}\Phi_{\rm H}(x,y)\big((\Delta+\kappa^2)g_{\rm H}(y)\big){\rm d}y\nonumber\\
&=-\int_{B_{R}}\big((\Delta+\kappa^2)\Phi_{\rm H}(x,y)\big)g_{\rm H}(y){\rm d}y\nonumber\\
&\quad-\int_{\partial B_R}\big(\Phi_{\rm H}(x,y)\partial_\nu g_{\rm H}(y)-g_{\rm H}(y)\partial_{\nu(y)} \Phi_{\rm H}(x,y)\big){\rm d} s_{y}\nonumber\\
&=-\int_{B_{R}}\big((\Delta+\kappa^2)\Phi_{\rm H}(x,y)\big)g_{\rm H}(y){\rm d}y+0=0,\quad |x|>R.
\end{align}
Combining \eqref{3DGM_ex} and \eqref{3DFM-7}, we have
\begin{align*}
u(x)=\frac{1}{2\kappa^2}(g_{\rm H}(x)-g_{\rm M}(x))=0,\quad |x|>R,
\end{align*}
which shows that $f(x)=(\Delta+\kappa^2)g_{\rm H}(x)$ is a nonradiating source for the three-dimensional biharmonic wave equation. 

Next, we demonstrate that $\alpha_n^m=\beta_n^m=0$ for all $n = 0,1,\dots, m = -n, \dots, n$. Noting 
\begin{align}\label{3Dbess-4}
\bigg[\frac{\rm d}{{\rm d} r}\bigg(r^2\frac{\rm d}{{\rm d} r}j_0(\kappa r)\bigg)+\kappa^2 r^2 j_0(\kappa r)\bigg]j_0^k(\kappa r)=0\times j_0^k (\kappa r)=0,\quad k\geq 2,
\end{align}
we get from the integration by parts that 
\begin{align}\label{3Dbess-7}
&-\int_{0}^{R}\big(\kappa^2 r^2 j_0(\kappa r)\big)j^k_{0}(\kappa r){\rm d}r
=\int_{0}^{R}\bigg[\bigg(\frac{\rm d}{{\rm d} r}\bigg(r^2\frac{{\rm d} j_{0}(\kappa r)}{{\rm d} r}\bigg)\bigg)j^k_{0}(\kappa r)\bigg]{\rm d}r\nonumber\\
&=\bigg(r^2\frac{{\rm d} j_{0}(\kappa r)}{{\rm d} r}\bigg)j^k_{0}(\kappa r)\bigg|_{0}^{R}
-\int_{0}^{R}\bigg[\bigg(r^2\frac{{\rm d} j_{0}(\kappa r)}{{\rm d} r}\bigg)\bigg(\frac{\rm d}{{\rm d} r}j^k_{0}(\kappa r)\bigg)\bigg]{\rm d}r\nonumber\\
&=-\int_{0}^{R}\bigg[\bigg(\frac{{\rm d} j_{0}(\kappa r)}{{\rm d} r}\bigg)\bigg(r^2\frac{\rm d}{{\rm d} r}j^k_{0}(\kappa r)\bigg)\bigg]{\rm d}r\nonumber\\
&=-j_{0}(\kappa r)\bigg(r^2\frac{\rm d}{{\rm d} r}j^k_{0}(\kappa r)\bigg)\bigg|_{0}^{R}
+\int_{0}^{R}\bigg[j_{0}(\kappa r)\bigg(\frac{\rm d}{{\rm d} r}\bigg(r^2\frac{\rm d}{{\rm d} r}j^k_{0}(\kappa r)\bigg)\bigg)\bigg]{\rm d}r\nonumber\\
&=-j_{0}(\kappa r)\bigg(r^2 k j_{0}^{k-1}(\kappa r)\frac{{\rm d} j_{0}(\kappa r)}{{\rm d} r}\bigg)\bigg|_{0}^{R}
+\int_{0}^{R}\bigg[j_{0}(\kappa r)\bigg(\frac{\rm d}{{\rm d} r}\bigg(r^2\frac{\rm d}{{\rm d} r}j^k_{0}(\kappa r)\bigg)\bigg)\bigg]{\rm d}r\nonumber\\
&=\int_{0}^{R}\bigg[j_{0}(\kappa r)\bigg(\frac{\rm d}{{\rm d} r}\bigg(r^2\frac{\rm d}{{\rm d} r}j^k_{0}(\kappa r)\bigg)\bigg)\bigg]{\rm d}r. 
\end{align}
Combining \eqref{3Dbess-4} and \eqref{3Dbess-7} leads to 
\begin{align}\label{3Dbess-8}
&\int_{0}^{R}\bigg[\bigg(\frac{\rm d}{{\rm d} r}\bigg(r^2\frac{\rm d}{{\rm d} r}\bigg)+\kappa^2 r^2\bigg)j_0^k (\kappa r)\bigg]j_{0}(\kappa r){\rm d}r\nonumber\\
&=\int_{0}^{R}\bigg[\frac{\rm d}{{\rm d} r}\bigg(r^2\frac{\rm d}{{\rm d} r}j_0(\kappa r)\bigg)+\kappa^2 r^2 j_0(\kappa r)\bigg]j_0^k(\kappa r){\rm d}r=0,\quad k\geq 2.
\end{align}
By \eqref{3Dbess-8}, for all $n = 0,1,\dots, m = -n, \dots, n$, we have 
\begin{align*} 
\alpha_n^m
&=\int_{0}^{R}\int_{0}^{\pi}\int_{0}^{2\pi} \Bigg[\frac{\frac{1}{r^2}\big(\frac{\rm d}{{\rm d} r}\big(r^2\frac{\rm d}{{\rm d} r}\big)+\kappa^2 r^2\big)j_0^{m_1}(\kappa r)}{\int_0^R j_0^{m_1}(\kappa r) j_0({\rm i}\kappa r)r^2{\rm d} r}-\frac{\frac{1}{r^2}\big(\frac{d}{d r}\big(r^2\frac{d}{d r}\big)
+\kappa^2 r^2\big)j_0^{m_2}(\kappa r)}{\int_0^R j_0^{m_2}(\kappa r)j_0({\rm i}\kappa r) r^2{\rm d} r}\Bigg]\nonumber\\
&\hspace{3cm}\times \overline{Y_{n}^{m}(\theta,\phi)} j_{n}(\kappa r)r^2\sin\theta{\rm d}\theta{\rm d}\phi{\rm d}r\nonumber\\
&= \frac{\int_{0}^{R}\big[\big(\frac{\rm d}{{\rm d} r}\big(r^2\frac{\rm d}{{\rm d} r}\big)
+\kappa^2 r^2\big)j_0^{m_1}(\kappa r)\big]j_{n}(\kappa r){\rm d}r\int_{0}^{\pi}\int_{0}^{2\pi}\overline{Y_{n}^{m}(\theta,\phi)} \sin\theta{\rm d}\theta{\rm d}\phi}{\int_0^R j_0^{m_1}(\kappa r) j_0({\rm i}\kappa r)r^2{\rm d} r}\nonumber\\
&\quad\quad-\frac{\int_{0}^{R}\big[\big(\frac{\rm d}{{\rm d} r}\big(r^2\frac{\rm d}{{\rm d} r}\big)
+\kappa^2 r^2\big)j_0^{m_2}(\kappa r)\big]j_{n}(\kappa r){\rm d}r\int_{0}^{\pi}\int_{0}^{2\pi}\overline{Y_{n}^{m}(\theta,\phi)} \sin\theta{\rm d}\theta{\rm d}\phi}{\int_0^R j_0^{m_2}(\kappa r)j_0({\rm i}\kappa r) r^2{\rm d} r}\nonumber\\
&= \frac{0}{\int_0^R j_0^{m_1}(\kappa r) j_0({\rm i}\kappa r)r^2{\rm d} r}
-\frac{0}{\int_0^R j_0^{m_2}(\kappa r)j_0({\rm i}\kappa r) r^2{\rm d} r}
=0.
\end{align*}

Noting 
\begin{align*}
\bigg[\frac{\rm d}{{\rm d} r}\bigg(r^2\frac{\rm d}{{\rm d} r}j_0({\rm i}\kappa r)\bigg)-\kappa^2 r^2 j_0({\rm i}\kappa r)\bigg]j_0^k(\kappa r) 
=0\times j_0^k(\kappa r)=0,
\end{align*}
for $k\geq 2$, we can similarly show 
\begin{align*}
\int_{0}^{R}\big(\kappa^2 r^2 j_0^k(\kappa r)\big)j_{0}({\rm i}\kappa r){\rm d}r
=\int_{0}^{R}\bigg[\frac{\rm d}{{\rm d} r}\bigg(r^2\frac{\rm d}{{\rm d} r}j_{0}^k(\kappa r)\bigg)\bigg]j_0({\rm i}\kappa r){\rm d}r,
\end{align*}
which implies
\begin{align}\label{3Dbess-2}
\int_{0}^{R}\bigg[\bigg(\frac{\rm d}{{\rm d} r}\bigg(r^2\frac{\rm d}{{\rm d} r}\bigg)+\kappa^2 r^2\bigg)j_0^k (\kappa r)\bigg]j_{0}({\rm i}\kappa r){\rm d}r
=2\kappa^2\int_{0}^{R}  j_0^k (\kappa r)j_0({\rm i}\kappa r)r^2{\rm d}r. 
\end{align}
Then, from \eqref{3Dbess-2}, we have
\begin{align*} 
\beta_n^m
&=\int_{0}^{R}\int_{0}^{\pi}\int_{0}^{2\pi} \Bigg[\frac{\frac{1}{r^2}\big(\frac{\rm d}{{\rm d} r}\big(r^2\frac{\rm d}{{\rm d} r}\big)+\kappa^2 r^2\big)j_0^{m_1}(\kappa r)}{\int_0^R j_0^{m_1}(\kappa r) j_0({\rm i}\kappa r)r^2{\rm d} r}-\frac{\frac{1}{r^2}\big(\frac{\rm d}{{\rm d} r}\big(r^2\frac{\rm d}{{\rm d} r}\big)
+\kappa^2 r^2\big)j_0^{m_2}(\kappa r)}{\int_0^R j_0^{m_2}(\kappa r)j_0({\rm i}\kappa r) r^2{\rm d} r}\Bigg]\nonumber\\
&\hspace{2cm}\times \overline{Y_{n}^{m}(\theta,\phi)} j_{n}({\rm i}\kappa r)r^2\sin\theta{\rm d}\theta{\rm d}\phi{\rm d}r\nonumber\\
&= \Bigg[\frac{\int_{0}^{R}\big[\frac{1}{r^2}\big(\frac{\rm d}{{\rm d} r}\big(r^2\frac{\rm d}{{\rm d} r}\big)
+\kappa^2 r^2\big)j_0^{m_1}(\kappa r)\big]j_{n}({\rm i}\kappa r)r^2{\rm d}r}{\int_0^R j_0^{m_1}(\kappa r) j_0({\rm i}\kappa r)r^2{\rm d} r}\\
&\quad\quad-\frac{\int_{0}^{R}\big[\frac{1}{r^2}\big(\frac{\rm d}{{\rm d} r}\big(r^2\frac{\rm d}{{\rm d} r}\big)
+\kappa^2 r^2\big)j_0^{m_2}(\kappa r)\big]j_{n}({\rm i}\kappa r)r^2{\rm d}r}{\int_0^R j_0^{m_2}(\kappa r)j_0({\rm i}\kappa r) r^2{\rm d} r}\Bigg]
 \int_{0}^{\pi}\int_{0}^{2\pi}\overline{Y_{n}^{m}(\theta,\phi)} \sin\theta{\rm d}\theta{\rm d}\phi,
\end{align*}
which implies 
\begin{align*}
\beta_n^m
&= \Bigg[\frac{\int_{0}^{R}\big[\frac{1}{r^2}\big(\frac{\rm d}{{\rm d} r}\big(r^2\frac{\rm d}{{\rm d} r}\big)
+\kappa^2 r^2\big)j_0^{m_1}(\kappa r)\big]j_{n}({\rm i}\kappa r)r^2{\rm d}r}{\int_0^R j_0^{m_1}(\kappa r) j_0({\rm i}\kappa r)r^2{\rm d} r}\nonumber\\
&\quad\quad
-\frac{\int_{0}^{R}\big[\frac{1}{r^2}\big(\frac{\rm d}{{\rm d} r}\big(r^2\frac{\rm d}{{\rm d} r}\big)
+\kappa^2 r^2\big)j_0^{m_2}(\kappa r)\big]j_{n}({\rm i}\kappa r)r^2{\rm d}r}{\int_0^R j_0^{m_2}(\kappa r)j_0({\rm i}\kappa r) r^2{\rm d} r}\Bigg]\times 0
=0,\quad n\neq 0,~m\neq 0
\end{align*}
and
\begin{align*}
\beta_0^0
&= \Bigg[\frac{\int_{0}^{R}\big[\frac{1}{r^2}\big(\frac{\rm d}{{\rm d} r}\big(r^2\frac{\rm d}{{\rm d} r}\big)
+\kappa^2 r^2\big)j_0^{m_1}(\kappa r)\big]j_{0}({\rm i}\kappa r)r^2{\rm d}r}{\int_0^R j_0^{m_1}(\kappa r) j_0({\rm i}\kappa r)r^2{\rm d} r}\\
&\quad \quad -\frac{\int_{0}^{R}\big[\frac{1}{r^2}\big(\frac{\rm d}{{\rm d} r}\big(r^2\frac{\rm d}{{\rm d} r}\big)
+\kappa^2 r^2\big)j_0^{m_2}(\kappa r)\big]j_{0}({\rm i}\kappa r)r^2{\rm d}r}{\int_0^R j_0^{m_2}(\kappa r)j_0({\rm i}\kappa r) r^2{\rm d} r}\Bigg] \int_{0}^{\pi}\int_{0}^{2\pi}\overline{Y_{0}^{0}(\theta,\phi)} \sin\theta{\rm d}\theta{\rm d}\phi\nonumber\\
&= 2\sqrt{\pi}\Bigg[\frac{\int_{0}^{R}\big[\big(\frac{\rm d}{{\rm d} r}\big(r^2\frac{\rm d}{{\rm d} r}\big)
+\kappa^2 r^2\big)j_0^{m_1}(\kappa r)\big]j_{0}({\rm i}\kappa r){\rm d}r}{\int_0^R j_0^{m_1}(\kappa r) j_0({\rm i}\kappa r)r^2{\rm d} r}\\
&\quad\quad\quad\quad
 -\frac{\int_{0}^{R}\big[\big(\frac{\rm d}{{\rm d} r}\big(r^2\frac{\rm d}{{\rm d} r}\big)
+\kappa^2 r^2\big)j_0^{m_2}(\kappa r)\big]j_{0}({\rm i}\kappa r){\rm d}r}{\int_0^R j_0^{m_2}(\kappa r)j_0({\rm i}\kappa r) r^2{\rm d} r}\Bigg]\nonumber\\
&= 2\sqrt{\pi}\Bigg[\frac{2\kappa^2\int_0^R j_0^{m_1}(\kappa r) j_0({\rm i}\kappa r)r^2{\rm d} r}{\int_0^R j_0^{m_1}(\kappa r) j_0({\rm i}\kappa r)r^2{\rm d} r}
-\frac{2\kappa^2\int_0^R j_0^{m_2}(\kappa r) j_0({\rm i}\kappa r)r^2{\rm d} r}{\int_0^R j_0^{m_2}(\kappa r)j_0({\rm i}\kappa r) r^2{\rm d} r}\Bigg]=0.
\end{align*}

\section{Conclusion}\label{S:co}

In this study, we undertake a thorough investigation on nonradiating sources of the two- and three-dimensional biharmonic wave equation. The goal is to reveal both their existence and inherent characteristics. Nonradiating sources, defined as sources that result in a field being identically zero beyond a finite region, share a close relationship with the issue of nonuniqueness in the inverse source problem. Specifically, we demonstrate that a source conforms to nonradiating behavior in the biharmonic wave equation if and only if it satisfies the condition of being a nonradiating source of both the Helmholtz equation and the modified Helmholtz equation. By utilizing the integral transform and the series representation of the fundamental solution, we deduce fundamental identities, which establish connections between the integral transform of the source and the near-field data through the integral representation of the solution for the biharmonic wave equation. We provide these conditions, which serve as valuable criteria for assessing the nonradiating attributes of sources, enabling researchers to facilitate the more efficient identification and characterization of nonradiating sources. Furthermore, we demonstrate that the integral operator possesses a nontrivial null space, constituting the collection of nonradiating sources. This observation holds implications for the nonuniqueness inherent in the inverse source problem.


\begin{thebibliography}{99}

\bibitem{AM2006}
R. Albanese and P. Monk, The inverse source problem for Maxwell's equations, Inverse Problems, 22 (2006), 1023--1035.

\bibitem{BCL2016}
G. Bao, C. Chen, and P. Li, Inverse random source scattering problems in several dimensions, SIAM/ASA J. Uncertain. Quantif., 4 (2016), 1263--1287.

\bibitem{BLLT2011}
G. Bao, P. Li, J. Lin, and F. Triki, Inverse scattering problems with multi-frequencies, Inverse Problems, 31 (2015), 093001.

\bibitem{BLZ2020}
G. Bao, P. Li, and Y. Zhao, Stability for the inverse source problems in elastic and electromagnetic waves, J. Math. Pures Appl., 134 (2020), 122--178.

\bibitem{BLT2010}
G. Bao, J. Lin, and F. Triki, A multi-frequency inverse source problem, J. Differential Equations, 249 (2010), 3443--3465.

\bibitem{BC1977}
N. Bleistein and J. K. Cohen, Nonuniqueness in the inverse source problem in acoustics and electromagnetics, J. Math. Phys., 18 (1977), 194--201.

\bibitem{B2018}
E. Bl\aa{}sten, Nonradiating sources and transmission eigenfunctions vanish at corners and edges, SIAM J. Math. Anal., 50 (2018), 6255--6270.

\bibitem{B2019}
E. Bl\aa{}sten and Y. Lin, Radiating and non-radiating sources in elasticity, Inverse Problems, 35 (2019), 015005.

\bibitem{CIL2016}
J. Cheng, V. Isakov, and S. Lu, Increasing stability in the inverse source problem with many frequencies, J. Differential Equations, 260 (2016), 4786--4804.

\bibitem{DR-2013}
D. Colton and R. Kress, Inverse Acoustic and Electromagnetic Scattering Theory, 3rd edition, Springer-Verlag, Berlin, 2013. 

\bibitem{EV2009}
M. Eller and N. Valdivia, Acoustic source identification using multiple frequency information, Inverse Problems, 25 (2009), 115005.

\bibitem{FGE2009}
M. Farhat, S. Guenneau, and S. Enoch, Ultrabroadband elastic cloaking in thin plates, Phys. Rev. Lett., 103 (2009), 024301.

\bibitem{GX2019}
Y. Gong and X. Xu, Inverse random source problem for biharmonic equation in two dimensions, Inverse Probl. Imaging,
13 (2019), 635--652.


\bibitem{KI1988}
K. Iwasaki, Scattering theory for 4th order differential operators: I, Japan. J. Math., 14 (1988), 1--57.

\bibitem{LLW2022}
J. Li, P. Li, and X. Wang, Inverse source problems for the stochastic wave equations: far-field patterns, SIAM J. Appl. Math., 82 (2022), 1113--1134.

\bibitem{LW2023}
P. Li and X. Wang, Inverse scattering for the biharmonic wave equation with a random potential, arXiv:2210.05900. 

\bibitem{LW2022}
P. Li and X. Wang, An inverse random source problem for the biharmonic wave equation, SIAM/ASA J. Uncertainty Quantification, 10 (2022), 949--974. 

\bibitem{LYZ2021IP}
P. Li, X. Yao, and Y. Zhao, Stability of an inverse source problem for the damped biharmonic plate equation, Inverse Problems, 37 (2021), 085003.

\bibitem{LYZ2021}
P. Li, X. Yao, and Y. Zhao, Stability for an inverse source problem of the biharmonic operator, SIAM J. Appl. Math.,
81 (2021), 2503--2525.

\bibitem{LY2017}
P. Li and G. Yuan, Increasing stability for the inverse source scattering problem with multi-frequencies, Inverse Probl. Imaging, 11 (2017), 745--759.

\bibitem{LZZ2020}
P. Li, J. Zhai, and Y. Zhao, Stability for the acoustic inverse source problem in inhomogeneous media, SIAM J. Appl. Math., 80 (2020), 2547--2559.

\bibitem{APS2000}
A. P. S. Selvadurai, Partial Differential Equations in Mechanics 2. The Biharmonic Equation, Poisson's Equation,
Springer-Verlag, Berlin, 2000.

\bibitem{SWW2012}
M. Stenger, M. Wilhelm, and M. Wegener, Experiments on elastic cloaking in thin plates, Phys. Rev. Lett., 108 (2012), 014301.

\bibitem{TS2018}
T. Tyni and V. Serov, Scattering problems for perturbations of the multidimensional biharmonic operator, Inverse Probl. Imaging, 12 (2018), 205--227.

\bibitem{WZC-2021}
B. Wang, W. Zhang, and W. Cai, Fast multipole method for 3-D Poisson-Boltzmann equation in layered 
electrolyte-dielectric media, J. Comput. Phys.,  (2021), 1--28.

\bibitem{YY2014}
Y. Yang, Determining the first order perturbation of a bi-harmonic operator on bounded and unbounded domains
from partial data, J. Differential Equations, 257 (2014), 3607--3639.

\end{thebibliography}
\end{document}